\documentclass[12pt]{article}
\usepackage{a4wide}
\usepackage{epsfig}
\usepackage{amsmath,amsfonts,amssymb,amsthm}
\usepackage{eucal}
\usepackage{graphics,graphicx}
\usepackage{float}
\usepackage{subfig}
\usepackage{epstopdf}
\usepackage{authblk}
\newtheorem{thm}{Theorem}[section]

\newtheorem{defn}{Definition}[section]
\newtheorem{lem}{Lemma}[section]
\newtheorem{rem}{Remark}[section]
\newcommand{\C}{\mathbb{C}}
\newcommand{\hC}{\widehat{\mathbb{C}}}

\newcommand{\Z}{\mathbb{Z}}

\begin{document}

\title{{Iteration of some topologically hyperbolic maps in the family $ \lambda+z+\tan z$}}

\author[2]{Subhasis Ghora\footnote{sg36@iitbbs.ac.in(Corresponding author)} 
\footnote{This work is mostly done when the corresponding author was an SRF at Indian Institute of \\ Technology Bhubaneswar}}
\author[1]{Tarakanta Nayak\footnote{tnayak@iitbbs.ac.in} }
\affil[1]{\textit{School of Basic Sciences \hspace{10cm}Indian Institute of Technology Bhubaneswar, India}  }
\affil[2]{\textit{Department of Mathematics \hspace{10cm}C. V. Raman Global University,  Bhubaneswar, India}  }
\date{}

\maketitle
\begin{abstract}

 Iteration of the function $f_ \lambda(z)=\lambda + z+\tan z, z \in \mathbb{C}$ is investigated in this article. It is proved that for every $\lambda$, the Fatou set of $f_\lambda$ has a completely invariant Baker domain $B$; we call it the primary Fatou component. The rest of the article deals with $f_\lambda$ when it is topologically hyperbolic.  For all real $\lambda$ or $\lambda$ such that $ \lambda= k\pi +i \lambda_2$ for some integer $k$ and $0 <  \lambda_2<1$, the only other Fatou component is found to be another completely invariant Baker domain. It is proved that if $|2+\lambda^2|<1$, then the  Fatou set is the union of $B$ and infinitely many invariant attracting domains. Every such domain $U$ has exactly one invariant access to infinity and is unbounded in a special way; $\{\Im(z): z\in U\}$ is unbounded whereas  for every $z_0\in U$,  $\{\Re(z): z \in U ~\mbox{and}~\Im(z)>\Im(z_0)\}$ is bounded. If $\Im(\lambda)> \sqrt{2}+ \sinh^{-1}1$ then it is found that the primary Fatou component is the only Fatou component and the Julia set is disconnected. For every natural number $k$, the Fatou set of $f_\lambda$ for $\lambda=k\pi+i\frac{\pi}{2}$ is shown to contain   $k$ wandering domains with distinct grand orbits. These wandering domains are found to be escaping. The Fatou set of $f_\lambda$ for $\lambda=k\pi+ i\frac{\pi}{2}$ is the union of $B$, these wandering domains and their pre-images.

	2010 Mathematics Subject Classification 37F50  30D05  30D030 
\end{abstract}

\section{Introduction}

A transcendental meromorphic map $f:\mathbb{C} \rightarrow \widehat{\mathbb{C}}$ with a single essential singularity is called general meromorphic if it has at least two poles or exactly one pole that is not an omitted value. We choose the essential singularity to be at $\infty$.  The Fatou set of $f$, denoted by $\mathcal{F}(f)$, is the set of all points in a neighborhood of which $\{f^n \}_{n>0}$ is defined and normal. Its complement in $\widehat{\mathbb{C}}$ is the Julia set of $f$ and it is denoted by  $\mathcal{J}(f)$.  For general meromorphic maps, the   backward orbit of $\infty$, $\{z: f^n(z)=\infty~\mbox {for~some~natural~number}~n\}$ is an infinite set and its closure turns out to be the Julia set of $ f$.  By the dynamics of a function, we mean its Fatou set and the Julia set.

%

\par
A maximally connected subset of the Fatou set is called a Fatou component.
For a given $n$, $U_n$ denotes the Fatou component containing $f^n(U)$. A Fatou component $U$ is said to be $p$-periodic if $p$ is the smallest natural number such that $U_p = U$. If $p=1$ then $U$ is called invariant. An invariant Fatou component $U$ is called completely invariant if $f^{-1}(U) \subseteq U$.  A periodic Fatou component can be an attracting domain, a parabolic domain, a rotational domain (a Herman ring or a Siegel disc) or a Baker domain. For a point $z_0$ if $p$ is the smallest natural number such that $f^p (z_0)=z_0$ then $z_0$ is called a $p-$periodic point of $f$.  A $1-$periodic point is called a fixed point.  An important number associated  with $z_0$ is its multiplier $\alpha_{z_0} =(f^p)'(z_0)$. The $p-$periodic point $z_0$ is called attracting, indifferent or repelling if $|\alpha_{z_0}|<1, =1$ or $>1$ respectively. An indifferent $p-$periodic point is called parabolic if $\alpha_{z_0}=e^{2 \pi i\beta}$ for some rational number $\beta$. A $p-$periodic attracting domain contains an attracting $p-$periodic point whereas  the boundary of a $p$-periodic parabolic domain contains a  parabolic $p-$periodic point. Similarly, a Siegel disc always contains a non-parabolic indifferent periodic point.  A $p-$periodic Fatou component $U$ is called a Baker domain if for some $U_k$,  $f^{np}(z)\rightarrow \infty$ uniformly on every compact subset of $U_k$.
A Fatou component $U$ is called wandering if $U_m\cap U_n=\emptyset$ for $m\neq n$.  Further details can be found in~\cite{milnor}.
\par 
The Newton method of $exp(-\int_{0}^z \frac{du}{i+\tan u})$ is the map $i+z +\tan z$, and it is reported in ~\cite{fagella2019,access} that this map has an invariant Baker domain  but has no wandering domain. 
It is proved  in ~\cite{access} that the upper half-plane is an invariant Baker domain for $z+\tan z$ and  the positive imaginary axis is an invariant, but not a strongly invariant  access to $\infty$. An access from a simply connected Fatou component $U$ to one of its boundary points $a$ is a homotopic class of curves in $U$ tending to $a$.  An  access is strongly invariant if it contains the image of each curve in it, in some way (for definition see Section 2). Rempe-Gillen and Sixsmith have recently shown that for $f(z)=z+\tan z$,
there are infinitely many disjoint simply connected domains $\{U_n\}_{n \geq1}$ such that $f^{-1}(U_n)$ is connected for all $n$~ \cite{vanilla}.
This gives a positive answer to a question raised by Eremenko:  Does there exist a non-constant meromorphic function having three disjoint simply-connected regions each with connected
pre-image? The above mentioned functions are two particular members of the one parameter family given by 
$$f_\lambda(z)=\lambda+z+\tan z~\mbox{for }~ \lambda \in \mathbb{C}.$$
\par  This article undertakes a systematic study of the Fatou set and the Julia set of  $f_\lambda$  for most of the values of $\lambda$.

\par
 A point $z$ is called a critical point of $f$ if $f'(z)=0$ and the image of a critical point  is known as a critical value of the function. A point $a\in \widehat{\mathbb{C}}$ is called an asymptotic value of $f$ if there exists a curve $\eta:[0,\infty)\rightarrow \mathbb{C}$ with $\lim_{t\rightarrow \infty} \eta(t)=\infty$ such that $ \lim_{t\rightarrow \infty} f(\eta(t))=a$. A subtle situation arises when the point at $\infty$ is an asymptotic value. The set of all the  singular values of $f$, denoted by  $S_f$ consists of  all the critical values,  asymptotic values and their limit points.  It is important to note that at every point of $S_f$, at least one branch of $f^{-1}$ fails to be defined. The union of the forward orbits of all the singular values is called the post-singular set of $f$. It is denoted by $P(f)$. More precisely, $$P(f)=(\cup_{s\in S_f}\cup_{n\geq0}f^n(s))\cap \mathbb{C}.$$ 

%
%
\par
 Most of the research on the dynamics of  general transcendental maps have been focussed on those with a bounded set of singular values; the set of all such functions is well-known as the Eremenko-Lyubich class. A Baker domain $U$ is special in the sense that  the essential singularity $\infty$ is always a limit function of $\{f^n\}_{n>0}$ on $U$. Every limit function of $\{f^n\}_{n>0}$ on a wandering domain is always constant and the set of all such limits can  be an infinte and unbounded set~\cite{bk2}. The Fatou set of  functions having only finitely many singular values cannot contain any Baker domain or any wandering domain. In order to have a Baker domain or a wandering domain, it is necessary for a map in the Eremenko-Lyubich class to have infinitely many singular values.   Several results on the relation of these types of Fatou components with the postsingular set are obtained in~\cite{fagella2019} though a complete understanding is yet to be arrived at. Some other aspects of dynamics of functions in the Eremenko-Lyubich class have also been investigated and a number of tools are developped.  However, the maps outside this class i.e., those with an unbounded set of singular values mostly remain unexplored. One of the motivations for taking up $f_\lambda(z)=\lambda+z+\tan z $ is that it is one such map. For suitable values of $\lambda$, the existence of Baker domain and wandering domain for $f_\lambda$ is established in this article.  

\par
The study of the dynamics of specific functions have been immensely useful,  not only for predicting results for a class of functions containing them but also often provides clues for their proofs. 
The first  general transcendental meromorphic map subjected to a systematic investigation from a dynamical point of view is probably  $z \mapsto \lambda\tan z$ for $\lambda \in \mathbb{C}$, which has only two singular values (in fact asymptotic values)~ \cite{keen-kot}. Later on, Sajid and Kapoor undertook the study of other maps including some with infinitely many singular values, namely $\lambda \frac{\sinh^2 z}{z^4}$ and $\lambda \frac{\sinh z}{z^2}$ ~\cite{sajid3, sajid4}. However, all these maps are in the Eremenko-Lyubich class. Nayak and Prasad investigated some meromorphic maps with an unbounded  set of singular values, namely $z \mapsto \lambda\frac{z^m}{\sinh^mz}$ for real $\lambda$ and the non-existence of Baker domain and wandering domain is established among other results in~\cite{nayakmgp} . 


\par
  
 A meromorphic map $f$ is said to be topologically hyperbolic if 
 $$d( {P(f)}, \mathcal{J}(f)\cap \mathbb{C})>0,$$ where $d$ is the Euclidean distance between two sets. For two subsets $A,B$ of $\mathbb{C}$, $d(A,B)$ is defined as the $\inf_{ a\in A, b \in B} |a-b|$.  The point at $\infty$ is allowed to be in $S(f)$ for topologically hyperbolic maps. All the finite points of the Julia set are uniformly away from the post-singular set.
 This article deals with  $f_\lambda$ that are topologically hyperbolic.
\par 
For real $\lambda$,  the Fatou set of $f_\lambda$ is the union of two completely invariant Baker domains. To see it, note that   $\Im(f_\lambda(z))>0$ (or $<0$) if and only if $\Im(z)>0$ (or $<0$ respectively)  for all $\lambda \in \mathbb{R}$. Therefore, the upper half-plane and the lower half-plane  are the two completely invariant Fatou components of $f_\lambda$,  by the Fundamental Normality Test (Lemma~\ref{FNT}). Since all the fixed points of $f_\lambda$ are real and repelling,  none of these Fatou components is  an attracting domain  or a parabolic domain. A completely invariant Fatou component cannot be a rotational domain and this gives that both the Fatou components are Baker domains.
Clearly, the extended real line $\mathbb{R} \cup \{\infty\}$ is the Julia set.  
\par

 The functions $f_\lambda$ and $f_{-\lambda}$ are conformally conjugate via $z \mapsto -z$, i.e.,$-f_{-\lambda}(-z)=-(-\lambda-z-\tan(z))=f_{\lambda}(z)$. This means that $-f_{-\lambda}^n (-z)=f_{\lambda}^n (z)$ for all $n$ and the dynamical behaviour (the Fatou and the Julia set) of $f_\lambda$ is essentially the same as that of $f_{-\lambda}$.  
In view of this, now onwards, we assume $\Im(\lambda) > 0$.

The following is a  straight forward observation and forms the basis of subsequent results. 

\begin{thm}\label{cifcupperhalf}
For  $\Im(\lambda) > 0,$ there is a   completely invariant Baker domain $B_\lambda$ of $f_\lambda$  containing the upper half-plane.

\end{thm}

We call the completely invariant Baker domain $B_\lambda$ of $f_\lambda,$ as the   \textit{ primary Fatou component} and denote it by $B$ whenever  $\lambda $ is understood.  Let us call a Fatou component \textit{ non-primary} if it is different from $B$.  Before looking into the non-primary Fatou components, we make  few remarks.
\begin{rem}\label{cifc-basic}
	\begin{enumerate}
		\item Since the Julia set is the boundary of every completely invariant Fatou component,   $\mathcal{J}(f_\lambda)=\partial B$.
		\item  Every Fatou component of $f_\lambda$ different from $B$  is simply connected. In particular, there is no Herman ring in the Fatou set of $f_\lambda$.
		\item All the critical points of $f_\lambda$ with positive imaginary part are in $B$ and it follows from the proof of Theorem \ref{cifcupperhalf} that 
		$d(\overline{P(f)}\cap  H^+, \mathcal{J}(f_\lambda)\cap \mathbb{C})>\sinh^{-1}1>0$ where $H^+=\{z: \Im(z) >0\}$.
	\end{enumerate}
\end{rem}
 
The function $f_\lambda$ has infinitely many fixed points for each $\lambda \neq  i$. These are the solutions of $ \tan z=-\lambda$. But the multiplier of each fixed point is $2+\lambda^2$ leading to some amount of advantage.
 First we consider  $|2+\lambda^2|<1$. The set of all such values of $\lambda$  in the upper half-plane is a bounded simply connected domain. The following theorem demonstrates that non-primary Fatou components do exist and it describes all of them.
  
\begin{thm}\label{attractingdomaincomplete}
Let $\mid2+\lambda^2\mid<1$. Then,
\begin{enumerate}

\item there are infinitely many invariant attracting domains of $f_\lambda$ and each such attracting domain $U$ is unbounded in such a way that $\{\Im(z): z \in U\}$ is unbounded and for every $z_0\in U$,  $\{\Re(z): z \in U ~\mbox{and}~\Im(z)>\Im(z_0)\}$ is bounded.  Further, there is exactly one invariant access from this attracting domain to $\infty$. 

\item   $f_\lambda$ does not have any other periodic Fatou component or any wandering domain.

\end{enumerate}
In other words, the Fatou set of $f_\lambda$ is the union of the primary Fatou component, all the invariant attracting domains and their pre-images.
\end{thm}

The images of attracting domains of $f_\lambda$ are shown in blue (see Figure \ref{attr-conjulia}). The primary Fatou components are indicated in yellow for $\lambda=0.1+i\frac{\pi}{2}$ and $-0.1+\frac{\pi}{2}$ in Figure \ref{attr-conjulia}(a) and Figure \ref{attr-conjulia}(c) respectively. It is shown in red for $\lambda=1.5i$ in Figure \ref{attr-conjulia}(b). The black horizontal line is the real axis in these figures.

\begin{figure}[H]
	\centering
	\subfloat[The attracting domains of $f_{0.1+i\frac{\pi}{2}}$ seen in blue.]
	{\includegraphics[width=1.86in,height=2in]{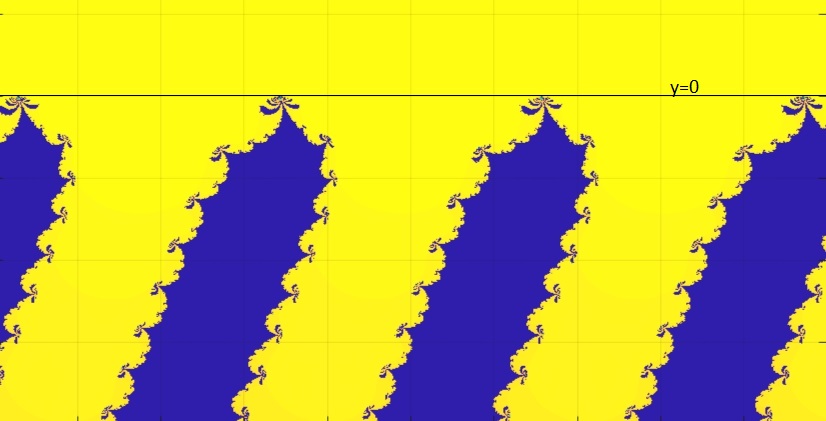}}
	\hspace{0.01in}
	\subfloat[The attracting domains of $f_{1.5 i}$ seen in blue.]
	{\includegraphics[width=1.86in,height=2in]{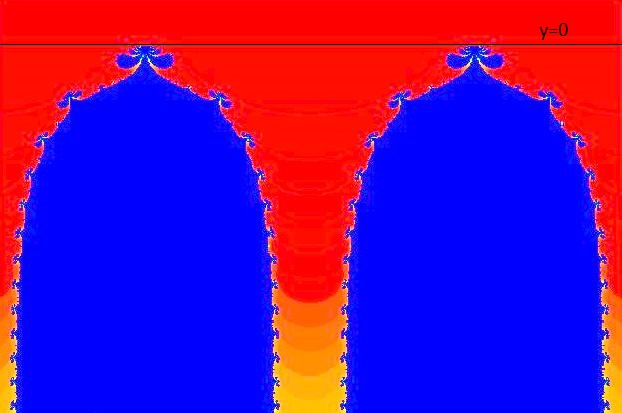}}
	\hspace{0.01in}
	\subfloat[The attracting domains of $f_{-0.1+i\frac{\pi}{2}}$ seen in blue.]
	{\includegraphics[width=1.86in,height=2in]{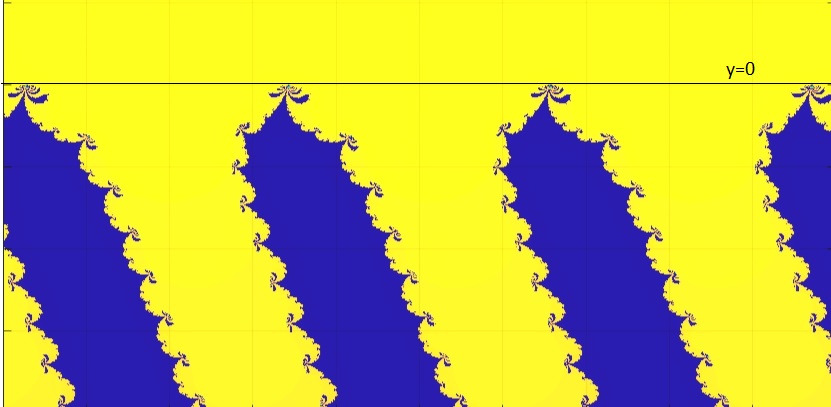}}
	\caption{Fatou sets of $f_\lambda$}	\label{attr-conjulia}\end{figure}

\begin{rem}
The boundary of the set $A=\{\lambda:\Im(\lambda)>0~\mbox{and }~ |2+\lambda^2|<1\}$ contains $i$ and $\sqrt{3}i$ and for every $\lambda \in A, 1 <  \Im(\lambda) < \sqrt{3}$. In particular, if $0< \Im(\lambda) <1$ or $\Im(\lambda)\geq \sqrt{2}+ \sinh^{-1} 1 > \sqrt{3}$ then all the fixed points of $f_\lambda$ are repelling. 
\label{lambda}
\end{rem}
Note that for a large set of parameters $\lambda$ (i.e., $|2+\lambda^2|>1$), all the fixed points of $f_\lambda$ are repelling and that calls for further effort to determine the dynamics. However, the situation is relatively simple if the imaginary part of such a parameter is either sufficiently large or sufficiently small. The following theorem makes it precise. 
\begin{thm}\label{repelling-one}
\begin{enumerate}
\item 
For $0<\Im(\lambda)<1$,  the Fatou set of $f_\lambda$   contains an invariant Baker domiain $\tilde{B}$ different from $B$. Further, if $\Re(\lambda) =k\pi$ for some integer $k$ then $\tilde{B}$ is the only non-primary Fatou component and the Julia set of $f_\lambda$ is connected.

\item 
For $\Im(\lambda)> \sqrt{2}+ \sinh^{-1}1$,  the primary Fatou component is the only Fatou component and the Julia set is not connected.
\end{enumerate}
\end{thm}
The Julia set of $f_\lambda$ for $\lambda={\pi+i(\sqrt{2}+\sinh^{-1}1)}$ is given as the complement of the yellow region-it is disconnected and is given in Figure~\ref{Fatouset-4}(a). The connected Julia set of $f_\lambda$ for $\lambda =\pi+ 0.99i$ is shown as the boundary of the yellow and the green region in Figure~\ref{Fatouset-4}(b).

\begin{figure}[H]
	\centering
	\subfloat[  $ \lambda=\pi+i(\sqrt{2}+\sinh^{-1}1)$]
	{\includegraphics[width=2.85in,height=2.8in]{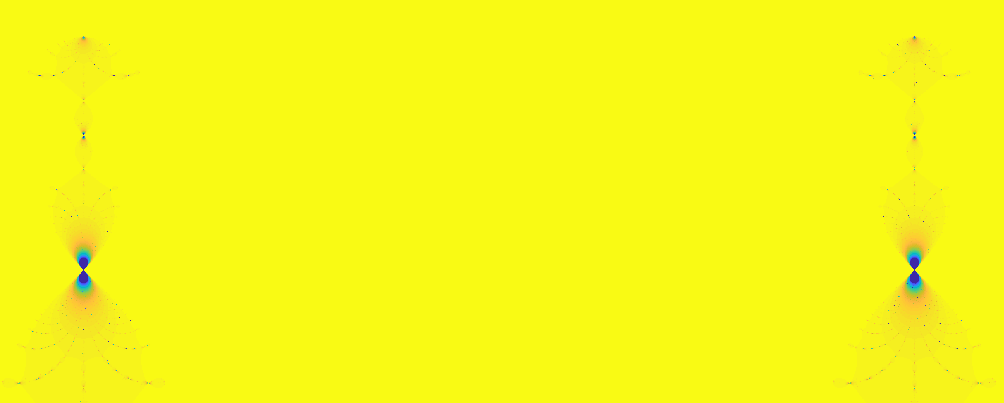}}
	\vspace{.2cm}
	\subfloat[$\lambda =\pi+0.99i$]	{\includegraphics[width=2.75in,height=2.8in]{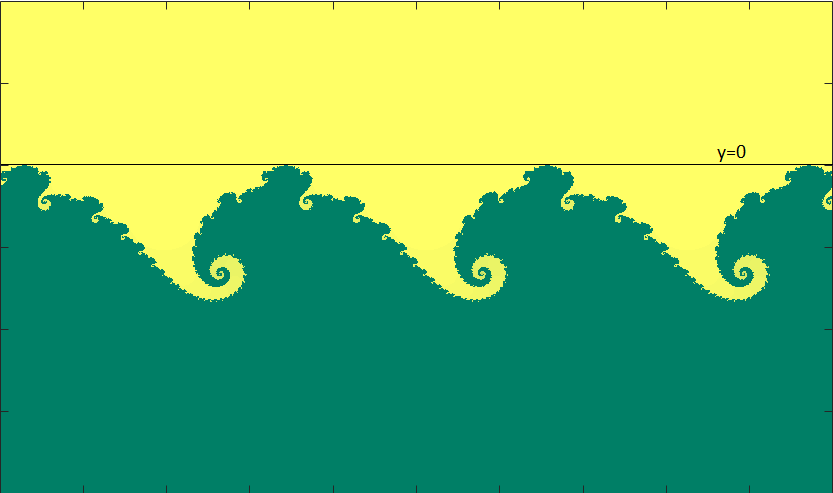}}
	\caption{The Julia sets of $f_\lambda$}
	\label{Fatouset-4}
\end{figure} 

\par
Every limit function  of $\{f^n\}_{n>0}$ on each  wandering domain of $f$ is always constant \cite{zheng}, one of which can be $\infty$.
For a wandering domain $W$, let  $L_{W}$ denote the set  of all limits of $\{f^n\}_{n>0}$ on $W$.  A wandering domain $W$ is called escaping if $L_W =\{\infty \}$. 

The following theorem proves the existence of escaping wandering domains for some values of $\lambda$ with $  \Im(\lambda)=\frac{\pi}{2}$. We say a Fatou component $U$ lands on a Fatou component $V$ if $U_n=V$ for some natural number $n$. The grand orbit of a wandering domain $W$ is the set of all wandering domains landing on $W$ or on one of its iterated forward images. Note that the grand orbit of two Fatou components are either identical or disjoint.
 
\begin{thm}\label{wandeing domain}
For every natural number $k$, there is a $\lambda$ such that  $f_\lambda$   has $k$ many  wandering domains with distinct grand orbits. If $W$ is such a wandering domain then  it has the following properties.

\begin{enumerate}
\item Each $W$ is escaping.

\item There is a two sided sequence of unbounded wandering domains 
$\{W_n \}_{n \in \mathbb{Z}}$ in the grand orbit of $W$ such that  $f_\lambda: W_{n} \to W_{n+1}$
 is a proper map with degree $2$. 
\item If  $W'$ is a wandering domain in the grand orbit of $W$ and  different from all $W_n$s then  $f_{\lambda}$ is one-one on $  W'  $.
\end{enumerate} 
The Fatou set is the union of the primary Fatou component and these $k$ many grand orbits of wandering domains. 
\end{thm}
The wandering domains of $f_\lambda$, $\lambda=\pi+i\frac{\pi}{2}$ are seen in green (see Figure \ref{wan1})
\begin{figure}[h]
	\centering
	\includegraphics[width=10cm,height=6cm,angle=0]{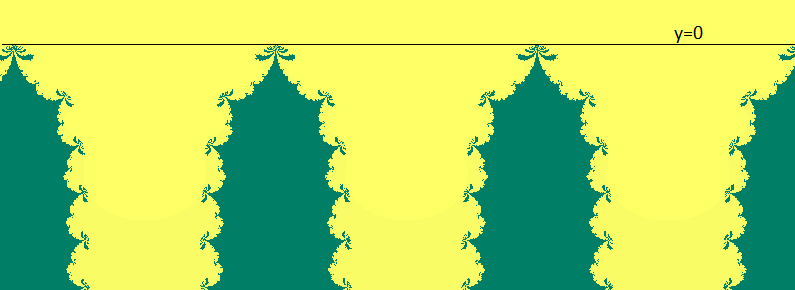}
	\caption{Wandering domains of $f_\lambda$ for $\lambda=\pi+i\frac{\pi}{2}$ in green.}
	\label{wan1}
\end{figure}
 
 For a complex number $z$, $\Im(z)$ and $\Re(z)$ denote the imaginary and the real part of $z$ respectively.
Let $H^+=\{z \in \C~:~\Im(z)>0\}$ and $H^-=\{z \in \C~:~\Im(z)<0\}$ be the upper and the lower half-plane respectively . For any set $A\subset \hC$, the boundary of $A$ is denoted by $\partial A$. For a complex number $w$, let $A+w=\{z+w:z \in A\}$. Let $D(a,r)$ denote the disc centered at $a$ and with radius $r$ and $\mathbb{D}$ denotes the unit disc. 
The set of all integers is denoted by $\mathbb{Z}$.

\section{Preliminaries}
 \subsection{Some useful results}
 We start with a useful result known as the Fundamental Normality Test.
 \begin{lem}\label{FNT}(Fundamental Normality Test)	If $f: \mathbb{C} \to \widehat{\mathbb{C}}$ is a meromorphic function and $D$ is a domain such that $\cup_{n>0} \{f^n(z): z \in D\}$ does not contain at least three points of $\widehat{\mathbb{C}}$ then   $\{f^n\}_{n>0}$ is normal in $D$.
 \end{lem}  
 %
A point $a$ on the boundary of a simply connected domain $U$ is called accessible from $U$ if there exists a curve $\gamma:[0,1]\rightarrow \widehat{\mathbb{C}}$ such that $\gamma([0,1)) \subset U$ and $\lim_{t\rightarrow 1^-}\gamma(t)=a$. We say that $\gamma$ lands at $a$.
There are simply connected domains such that a  point on its boundary is not accessible. In particular,  $\lim_{t_n \to 1-} \gamma(t)$ may be different for different sequences $t_n$ converging to $1$ from the left hand side. Some such examples can be found in ~\cite{milnor}. For an accessible point, there are uncountably many curves landing on it. What is important is the set of homotopically equivalent classes of such curves.
\begin{defn}{ {(Access)}}
	For a simply connected domain $U$, let $z_0 \in U$ and  $a\in \partial U$ be an accessible point. An access $\mathcal{A}$ from $U$ to $a$ is the class of all curves $\gamma:[0,1]\rightarrow \widehat{\mathbb{C}}$ homotopic to each other such that $\gamma([0,1)) \subset U$, $\gamma(0)=z_0$ and  $\lim_{t\rightarrow 1^-}\gamma(t)=a$.
\end{defn}  
Accesses on simply connected Fatou components of a meromorphic function are our concern.
\begin{defn}{ {(Invariant and strongly invariant access)}} Let $U$ be a simply connected and invariant Fatou component of a meromorphic function $f$.  An access
	$\mathcal{A}$  from $U$  to one of its boundary points $a$ is called invariant if there exists $\gamma \in \mathcal{A}$ such that $f(\gamma) \cup \gamma_1 \in \mathcal{A}$, where $\gamma_1: [0,1] \to U$ is a curve contained in $ U$ such that  $\gamma_1(0)=z_0$ and $\gamma_1(1)=f(z_0)$. If $f(\gamma) \cup \gamma_1 \in \mathcal{A}$  for every $\gamma \in \mathcal{A}$ then $\mathcal{A}$  is called a strongly invariant access.
\end{defn}

For an invariant simply connected Fatou component $U$ of $f$, if $\phi :\mathbb{D}\rightarrow U$ is the Riemann map then the inner function $g:\mathbb{D}\rightarrow \mathbb{D}$ associated with $f$ is defined as $g=\phi^{-1}\circ f\circ \phi$. 
We need the following result (Theorem B,~\cite{access}) relating the behaviour of $g$ on the unit circle to that of $f$ on the boundary of $U$. A fixed point of $f$ is called weakly repelling if it is either repelling or is parabolic with multiplier equal to $1$.
\begin{thm}\label{accesstoinfty}
	Let $U$ be a simply connected and invariant  Fatou component of $f$ and $g=\phi^{-1}\circ f\circ \phi$ be the inner function associated with $f|_U$. If  the degree $d$ of $f$ on $U$ is finite and $d_1$ is the number of fixed points of $g$ in $\partial\mathbb{D}$ then $f$ has exactly $d_1$  invariant accesses, and    $d-1 \leq d_1\leq d+1$. Moreover, every invariant access of $f$ from $U$ either lands at $\infty$ or at a weakly repelling fixed point of $f$.
	
\end{thm}

%

 Recall that the post singular set of $f$, denoted by  $P(f)$ is $$P(f)=(\cup_{s\in S_f}\cup_{n\geq0}f^n(s))\cap \mathbb{C}.$$   Here is a well-known result.
\begin{lem}\label{singularvalues}
Every attracting domain and parabolic domain of a meromorphic function intersects  $S_f$. If  $U$ is a rotational domain then  $\partial U \subset  \overline{P(f)}$. In particular, the Fatou set of a topologically hyperbolic map cannot contain any rotational domain.
\end{lem} 
The following lemma  proved in \cite{fagella2019} reveals the connection of the singular values with the Fatou components. In particular, this is more relevant for Baker domains and wandering domains for topologically hyperbolic meromorphic maps.
%
%

\begin{lem}
\label{wandeingarbitrarylagredisc}

Let $U$ be a Fatou component of a topologically hyperbolic meromorphic map $f$ such that  $U_n \cap P(f)=\emptyset$ for all $n >0$. Then for every compact set $K \subset U$ and every $r>0$, there exists $n_0$ such that for every $z \in K$ and every $n \geq n_0$, $D(f^n(z),r) \subset U_n$. 

\end{lem}
We end this subsection by stating a very important result. For a continuous map $f: V \rightarrow U$ between two open connected subsets of $\mathbb{C}$ if the pre-image of each compact subset of $U$ is compact in $V$ then $f$ is called proper. Further, if $f$ is analytic then there is a $d$ such that every element of $U$ has $d$ pre-images counting multiplicity. Here, the multiplicity of a point $z$ is the local degree of $f$ at $z$. This number $d$ is known as the degree of  $f: V \to U$.
The following lemma proved in \cite{bolsch} is to be applied repeatedly.
\begin{lem}\label{RH}{ {(Riemann-Hurwitz formula)}}

Let $f:\mathbb{C}\rightarrow \widehat{\mathbb{C}}$ be a transcendental meromorphic function. If $V$ is a  component of the pre-image of  an open connected set $U$ and $f : V \rightarrow U$ is a proper map of degree $d$, then $c(V)-2=d(c(U)-2)+n$, where $n$ is the number of critical points of $f$ in $V$ counting multiplicity and $n \leq 2d-2$. Here, the multiplicity of a critical point is one less than the local degree of $f$ at the critical point.  
\end{lem}

\subsection{Some basic properties of $f_\lambda$}
We make few preliminary observations on  $f_{\lambda}(z)=\lambda+z+\tan z$ for $\Im(\lambda)>0$.
First note that $\tan (z+\pi)=\tan z$ for all $z$ and  for $z=x+iy$, $$\Re(\tan z) = \frac{\sin 2x}{\cos 2x+\cosh 2y}~~~\mbox{and} ~~~\Im(\tan z) = \frac{\sinh 2y}{\cos 2x+\cosh 2y}.$$

\begin{lem}\label{symmetry}
	The Fatou set $\mathcal{F}(f_\lambda)$ is invariant under $z \mapsto z+\pi$ i.e., $z \in \mathcal{F}(f_\lambda)$ if and only if $z+\pi \in \mathcal{F}(f_\lambda)$. If a Fatou component $U$ contains a point $z$ and its $k\pi-$translate $z+k \pi$ for some non-zero $k \in \mathbb{Z}$ then $\{\Re(z): z \in U\}=\mathbb{R}$. In particular, this is true if $U$ contains a horizontal line segment of length bigger than $\pi$. 
\end{lem}

\begin{proof}
Since $f_{\lambda}(z+\pi)=f_{\lambda}(z)+\pi $,  $f^n_{\lambda}(z+\pi)=f^n_{\lambda}(z)+\pi $ for all $n$. Therefore $z\in\mathcal{F}(f_{\lambda})$ if and only if $z+\pi \in\mathcal{F}(f_{\lambda})$. 
If a Fatou component $U$ contains $z$ as well as $z+k \pi$ for some non-zero integer $k$ then for a curve $\gamma \subset U$ joining and containing  these two points, we have $\cup_{n \in \mathbb{Z} } \gamma+ n \pi \subset U$. Thus $\{\Re(z): z \in \cup_{n \in \mathbb{Z}} \gamma+ n \pi\}=\mathbb{R}$.
\end{proof}
The following describes the behaviour of $f_\lambda$ on some vertical lines. For a vertical line $l$ and a real number $r$, let $l+r=\{z+r: z\in l\}$.

\begin{lem}\label{lines} Let $m $ be an integer and  $l_{m \pi} =\{z: \Re(z)=m \pi\}$ .
	\begin{enumerate}
		\item The function $f_\lambda$ maps the line $l_{m \pi}  $   bijectively onto $l_{m \pi +\Re(\lambda)}$.
		
		\item If $\lambda=k \pi+i \lambda_2$ for some $k \in \mathbb{Z}$ and $\lambda_2 >1$ then $\lim_{n \to \infty} \Im(f_{\lambda}^n (z))=+\infty$ for all $z \in l_{m \pi}$.
	\end{enumerate}	 
\end{lem}
\begin{proof} For $z=m \pi +iy, f_{\lambda}(z)=\lambda+ m \pi +iy+i \tanh y $. Define $\phi: \mathbb{R} \to \mathbb{R}$ by $\phi(y)=\Im(\lambda)+  y+\tanh y$. This is a strictly increasing function satisfying $\lim_{y \to -\infty} \phi(y)=-\infty$ and $\lim_{y \to \infty} \phi(y)=\infty$. In particular, this is a bijection.
	\begin{enumerate}
		 
\item Since   $\phi(y)$ is a bijection of the real line onto itself,  $f_\lambda$ maps $l_{m \pi}$ bijectively onto  $l_{ m\pi +\Re(\lambda)}$. 
\item  For $\lambda=k \pi+i \lambda_2,~ k \in \mathbb{Z}$ and $\lambda_2 >1$, $\phi(y)=\lambda_2+  y+\tanh y>y$ for all $y$. This (non-existence of any fixed point) along with the strict increasingness of $\phi$ implies that $\lim_{n \to \infty}\phi^n (y) =+\infty$. Since $\Im(f_{\lambda}(m \pi+iy))=\phi(y), \Im(f_{\lambda}^2(m \pi+iy))=\phi^2(y)$ and in general, $\Im(f_{\lambda}^n(m \pi+iy))=\phi^n(y)$ for all $n>0$,  $\lim_{n \to \infty} \Im(f_{\lambda}^n (z))=+\infty$ for all $z \in l_{m \pi}$.
	\end{enumerate}
\end{proof}

To determine all the singular values of $f_\lambda$, let  $\overline{C}$ denote the set $\{\overline{z}: z \in C\}$ whenever  $C $ is a set of complex numbers. Recall that we have assumed $\Im(\lambda)>0$.
\begin{lem}\label{criticalpoint}
	\begin{enumerate}
\item The set of all critical points of $f_\lambda$ is $C \cup \overline{C}$ where $C=\{\frac{\pi}{2}+n \pi  +i \sinh^{-1}1 : n \in \mathbb{Z}\}$. The critical values are $ 	\lambda+\frac{\pi}{2}+n\pi\pm i(\sinh^{-1}1+\sqrt{2})$ where $n \in \mathbb{Z}$.  
		
\item The point at infinity is the only asymptotic value of $f_\lambda$ and there is only one transcendental singularity lying over it.
	\end{enumerate}
\end{lem}

\begin{proof}
	\begin{enumerate}
		\item 
The solutions of $f'_{\lambda}(z)=0$ are precisely those satisfying $\cos z=i\ \mbox{or}\ -i$. Since $\overline{\cos z}=\cos \overline{z}$ for all $z \in \mathbb{C}$, we have $\cos z=i$ if and only if $\cos \overline{z}=-i$.
		\par 
		Let $\cos z=i$. Then 
	$
		\cos x \cosh y- i \sin x\sinh y=i.$
	As $\cosh y$ is never zero,  $\cos x=0$ and $\sin x\sinh y=-1$. The first equation gives that $x=x_n=\frac{\pi}{2}+n \pi $ for all $n\in\mathbb{Z}$.  If $n$ is odd then $\sin x_{n}=-1$ and $ \sinh y=1$ and, the solution is   $\frac{\pi}{2} +n \pi+ i \sinh^{-1}1$. Similarly for even  $n$, $\sin x_{n}=1,\ \sinh y=-1$
and we have $    \frac{\pi}{2} +n \pi+i \sinh^{-1}(-1)$ as the solution of $\cos z=i$. Taking the complex conjugate of these solutions, the set of all critical points of $f_\lambda$ is now found to be $C \cup \overline{C}$ where $C=\{c_n =\frac{\pi}{2}+n \pi  +i \sinh^{-1}1 : n \in \mathbb{Z}\}$. Since $\tan c_n =i \coth(\sinh^{-1} 1)=i \sqrt{2}$, $f_{\lambda}(c_n)= \lambda+\frac{\pi}{2}+n\pi+i \sinh^{-1} 1+\tan (c_n)=	\lambda+\frac{\pi}{2}+n\pi+i(\sinh^{-1}1+\sqrt{2})$. Similarly $f_{\lambda}(\overline{c_n})= \lambda+\frac{\pi}{2}+n\pi-i(\sinh^{-1}1+\sqrt{2})$.
 
		
		\item

\par Let  $D$ be a disc centered at $\infty$ with respect to the spherical metric. Then there exists a $\delta>0$ such that the
  half-planes  $H_{\delta} =\{z: \Im(z)> \delta\}$ and $\overline{H_{\delta}}=\{z: \Im(z) < -\delta \}$ are  contained in $D$. Since  $H_{\delta}$ is invariant under $f_\lambda$ (as $\Im(\lambda)>0$), $f_{\lambda}^{-1}(D)$ contains  
$H_{\delta}$. Note that if $\Im(z) <-\delta-\Im(\lambda)$ then $\Im(f_\lambda) <-\delta +\Im(\tan z) < -\delta$. In other words, the half-plane $H_{-\delta-\Im(\lambda)}=\{z: \Im(z) < -\delta-\Im(\lambda)\}$ is mapped into $\overline{H_\delta} \subset D$ giving that $H_{-\delta-\Im(\lambda)} \subset f_{\lambda}^{-1}(D)$. Therefore,
 \begin{equation}
\label{halfplanes}
H_\delta \cup H_{ -\delta-\Im(\lambda)} \subset f^{-1}_{\lambda}(D).\end{equation}

The disc $D$ contains the left half-plane $H_\alpha=\{z: \Re(z)< \alpha\}$ for some $\alpha < 0$ and its reflection about the imaginary axis  $-H_{\alpha}=\{ z: \Re(z) > - \alpha\}$.  	There is a natural number $m_0$ (depending on $\alpha$ and $\lambda$) such that the line $l_{m \pi +\Re(\lambda)}=\{z: \Re(z)=m\pi +\Re(\lambda)\}$ is contained in $D$ for all integers $m$ with  $|m|>m_0$. By Lemma~\ref{lines}(1), we have  
\begin{equation}\label{line}
 l_{m \pi} =\{z: \Re(z)=m \pi\} \subset f^{-1}_{\lambda}(D) ~\mbox{ for infinitely many values of}~m.
\end{equation}

Now it follows from Equation(\ref{halfplanes}) and Equation(\ref{line}) that there is a unique unbounded component $D_{-1}$ of $f_{\lambda}^{-1}(D)$ such that each component of $\widehat{\mathbb{C}}\setminus D_{-1}$ is bounded. In other words, for every unbounded curve $\gamma:[0,1)\rightarrow \mathbb{C}$ with $\lim_{t \to 1^-}\gamma(t)=\infty$, $f_\lambda(\gamma(t))$ has an accumulation point at $\infty$. This along with the fact that $\lim_{t\rightarrow 1^-}f_\lambda(\frac{i}{t-1})=\infty$ gives  that $\infty$ is the only asymptotic value of $f_\lambda$. It is also clear that there is only one singularity lying over $\infty$.

	\end{enumerate}
\end{proof}
\begin{rem}\label{criticalpoints-rem}
	\begin{enumerate}
		\item All the critical points are simple, i.e., the local degree of $f_\lambda$ is two at every critical point.
		\item  Note that $C \subset H^+$ and $\overline{C} \subset H^-$. The critical values corresponding to the critical points belonging to $C$ are in $H^+$ whenever $\Im(\lambda)>0$. The other critical values are on the same horizontal line but may not be in $H^+$.
	\end{enumerate}

	\end{rem}

Now we determine some properties of the fixed points of $f_\lambda$.
\begin{lem}\label{multiplierofthefixedpoints}
	For each $\lambda \neq i$ with $\Im(\lambda) > 0$, $f_\lambda$ has infinitely many fixed points. Moreover, the following are true.
	\begin{enumerate}
		\item The multiplier of each fixed point is $2+\lambda^2$. In other words, all the fixed points of $f_\lambda$ are attracting,  repelling or indifferent  together.
		\item A point $z$ is a fixed point of $f_\lambda$ if and  only if $z+n \pi$ is so for all $n \in \mathbb{Z}$.
		\item   All the fixed points of $f_\lambda$ are in $H^{-}$ whenever $\Im(\lambda)>0$.
	\end{enumerate}
\end{lem} 

\begin{proof}
	\begin{enumerate}
		\item The fixed points  of $f_\lambda$ are the solutions of $\tan z=-\lambda$. Since $\lambda \neq i$, there are infinitely many fixed points. The multiplier of each fixed point is $f'_{\lambda}(z)=1+\sec^2 z =2+\lambda^2$. It  depends on the value of $\lambda$ but not on any particular fixed point. 
		All the fixed points are attracting, repelling or indifferent if and only if $|2+\lambda^2|<1,~>1~\mbox{or}~=1$ respectively.
		\item This follows from the fact that $\tan z$ is $\pi-$periodic.
		\item The is so because all the solutions of $\tan z=-\lambda, \Im(\lambda)>0$ are in $H^{-}$.
		
	\end{enumerate}
	
	%
\end{proof}

\begin{rem}
The fixed points of $f_\lambda$ are real  if and only if $\lambda$ is real.
\end{rem}

%
%

\section{The proofs}

Here is the proof of Theorem~\ref{cifcupperhalf}. 

\begin{proof}[Proof of Theorem \ref{cifcupperhalf}]
	Note that for all $z \in H^+$, $\Im(f_\lambda(z))>\Im(\lambda)+\Im(z)>0$. The family $\{f^n\}_{n \geq 0}$ is normal in $H^+$ by the Fundamental Normality Test. Since $\Im(f^{n}_\lambda(z))>n \Im(\lambda)+\Im(z)$ for all $n$,  $f^n_\lambda(z)\rightarrow \infty$ as $n\rightarrow \infty$ for all $z \in H^+$. Thus $f_\lambda$ has an invariant Baker domain containing the upper half-plane. This is the primary Fatou component and we denote it by $B$. 
	\par In order to show that $B$ is backward invariant, let 
	$B_{-1}$ be a component of $f_{\lambda}^{-1}(B)$. It is known that if $U$ and $V$ are two Fatou components of a meromorphic function $f$ such that $f: U \to V$ then $V \setminus f(U)$ contains at most two points~(Theorem 1, \cite{herring}). Therefore, $B \setminus f_{\lambda}(B_{-1})$ contains at most two points. 
	Consider $  \epsilon_1 < \epsilon_2   $ and the  horizontal line segment $l=\{w:  \epsilon_1 \leq \Re(w) \leq \epsilon_2 ~\mbox{and}~ \Im(w)=\Im(\lambda)\}  \subset f_{\lambda}(B_{-1})$. Note that  $ l \subset H^+ \subset B$. For $w \in l$, let $z$ be such that $\lambda +z+\tan z=w$. Since $\Im(w)=\Im(\lambda)$, $\Im(z)+\Im(\tan z)=0$. Note that $\Im(z)=0$ if and only if $\Im(\tan z)=0$ and hence, $\Im(z)+\Im(\tan z)=0$ gives that $z$ is a real number. Each real number except the poles of $f_\lambda$ is mapped into  $H^+$ by $f_\lambda$ and therefore $B$ contains the real line except the poles. Thus the full pre-image $f_\lambda^{-1}(l)$ of $l$ is contained in $B$. On the other hand the set $B_{-1}$ intersects $f_\lambda^{-1}(l)$ which gives that $B_{-1}$ intersects $B$.  Thus $B$ is backward invariant. Therefore $B$ is a completely invariant Baker domain.
	
\end{proof}

The following lemma states that the set of all pre-images of every point in the lower half-plane is spread horizontally.
\begin{lem}\label{nopre-image}
For $\Im(\lambda)>0$ and $w \in H^-$, if $f_{\lambda}(z) =w$ then  $\Im(z)  > \Im(w)-\Im(\lambda)$. 
\end{lem}

\begin{proof}
If $f_{\lambda}(z)=w$ then $z \in H^{-}$ (because $f_{\lambda}(H^+) \subset H^+$) and $\Im(\tan z) <0$.  Now, 
 $\Im(w)= \Im(\lambda)+\Im(z)+\Im(\tan z)<\Im(\lambda)+\Im(z)$. This is what is claimed. 
\end{proof}

Every point in a non-primary Fatou component has negative imaginary part. Note that a Fatou component containing the image or any pre-image of a non-primary Fatou component is also non-primary.  A non-primary Fatou component $U$ is called \textit{horizontally spread} if there is a $\delta < 0$ such that $\{\Re(z): z \in U~\mbox{and }~\Im(z)>\delta\}$ is unbounded. Horizontally spread Fatou components are unbounded in a special way. The existence of a sequence of points $z_n$ in $U$ with $\Im(z_n) \to -\infty$ as $n \to \infty$ is not ruled out and $U$ is allowed to contain even a half-plane of the form $\{z: \Im(z) < \delta'\}$ for some $\delta' <0$. The following describes some useful properties of horizontally spread Fatou components that are to be used in the proof of Theorem~\ref{attractingdomaincomplete}.
\begin{lem}\label{proper}
For  $\Im(\lambda)>0$, let $U$  be a non-primary Fatou component  of $f_\lambda$.
\begin{enumerate}
	\item 
 If $U$ is horizontally spread and is not  invariant under $z \mapsto z+\pi$ then $f_\lambda$ has an invariant Baker domain.
 \item If $U$ is not horizontally spread then 
 $f_\lambda:U \rightarrow U_1$ is a proper map with degree $   1$ or $2$. 
   
\end{enumerate}
\end{lem}

\begin{proof}
  
\begin{enumerate}
	\item If $U$ is horizontally spread  then all its $k \pi$-translates $U+k\pi =\{z+ k\pi: z \in U\}$ are also horizontally spread.  Since $U$ is not invariant under $z \mapsto z+\pi$, $U+k \pi \cap U+k' \pi =\emptyset$ for all $k \neq k'$ (by Lemma~\ref{symmetry}). Now, if $\{\Im(z): z \in U\}$ is unbounded then we can find an unbounded Jordan curve $\gamma \subset U \cup \{\infty\}$ which separates the primary Fatou component $B$ from $U'$ where $U'=U+\pi$ or $U -\pi$, i.e., one component  of $\widehat{\mathbb{C}} \setminus \gamma$, say $B'$ contains $B$ whereas the other contains $U'$. This means that  $\mathcal{J}(f_\lambda)=\partial B $ is contained in the closure of $B'$ which contradicts the fact that the other component of $\widehat{\mathbb{C}} \setminus \gamma$ contains some points of the Julia set, namely those on the boundary of $U'$. The two possibilities of this situation are given in the Figure \ref{cross}(a).  Thus, the set  $\{\Im(z): z \in U\}$ and therefore $\{\Im(z): z \in U+k \pi\}$ for all $k$ is bounded. Using the same argument, it can be seen that $\{\Re(z): z \in U\} =\mathbb{R}$ is not possible and there is a $\delta$ such that $\Re(z)>\delta $ or $\Re(z)< \delta $ for all $z \in U$. We assume that $\Re(z)>\delta $ for all $z \in U$. Note that for each $k<0$, $U+k\pi$ is horizontally spread and $\{\Im(z):z \in U+k\pi\}$ is bounded. If $\Re(z)<\delta$ for all $z\in U$ then we have to consider $U+k\pi$ for all $k>0$ (left part of Figure \ref{cross}(b)). The right part of Figure \ref{cross}(b) shows the situation when $\Re(z)>\delta$ for all $z \in U$ and in that case $k<0$ is to be taken. 
	\begin{figure}[H]
		\centering	\subfloat[Image of $U,U+\pi, U-\pi$]
		{\includegraphics[width=2.85in,height=2.4in]{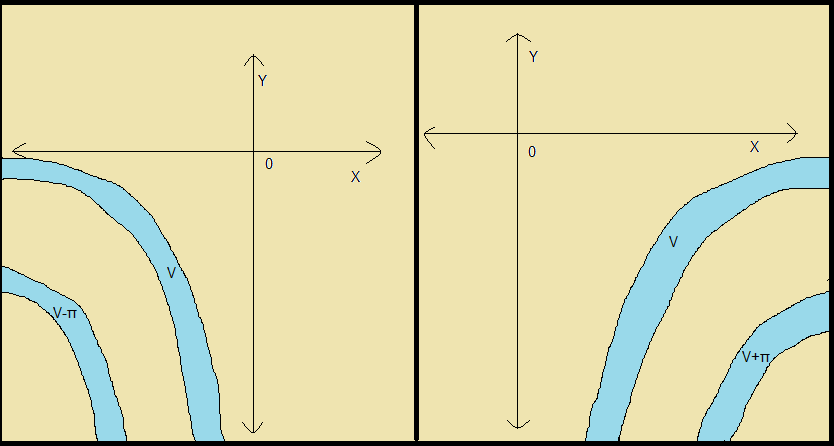}}
		\vspace{.2cm}
		\subfloat[Image of $U+k\pi$'s]	{\includegraphics[width=2.75in,height=2.4in]{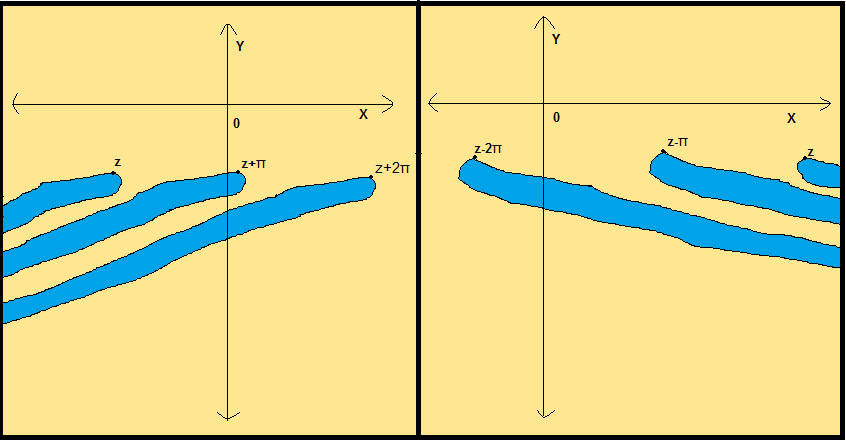}}
		\caption{Horizontally spread Fatou components.}
		\label{cross}
	\end{figure}
	\par  Let $\partial_k$ be the boundary of $U+k \pi$ and $\alpha$ be the set of all the limit points of $\partial_k$, i.e., $\alpha=\{z: ~\mbox{there is a sequence}~ z_{k_n} \in \partial_k~\mbox{such that}~\lim_{n \to \infty}z_{k_n}=z\}$. This $\alpha$ is an unbounded connected subset of the Julia set. Further, $\{\Re(z): z\in \alpha\}=\mathbb{R}$ and $\{\Im(z): z\in \alpha\}$ is bounded. Now one component of $\widehat{\mathbb{C}} \setminus \alpha$ contains the primary Fatou component $B$ and the other component  must be a Fatou component, say $\tilde{B}$. This $\tilde{B}$ contains a lower half-plane $H_\beta=\{z: \Im(z)< \beta\}$ for some $\beta<0$. Since $\Im(f_{\lambda}(z))=\Im(\lambda)+\Im(z)+\Im(\tan z)$, we can choose a $z \in \tilde{B}$ (depending on $\lambda$) with imaginary part sufficiently near to $-\infty$ such that its image is in $\tilde{B}$. For example, take $\beta_1 < \beta - \Im(\lambda)$ such that $\Im(\tan z) \in (-1.1,-0.9)$ for $\Im(z)< \beta_1$ and note that $\Im(f_\lambda (z))< \beta -0.9 < \beta$. This shows that  $\tilde{B}$ is invariant. If $\lim_{n \to \infty}f^n _{\lambda}(z) $ is a fixed point $z_0$ for 
	some $z \in H_{\beta_1} \subset  \tilde{B}$ then  $z+\pi \in H_{\beta_1}$ and $\lim_{n \to \infty}f^n _{\lambda}(z+\pi) $ is $z_0 +\pi$, which is also a fixed point. This cannot be true if $\tilde{B}$ is either an attracting domain or a parabolic domain. Similarly, it can  be seen that it is also not a Siegel disc. Therefore $\tilde{B}$ is a Baker domain.  
 	
\item If $U$ is not horizontally spread then it follows from Lemma~\ref{nopre-image} that every point of $U_1$, the Fatou component containing $f_{\lambda}(U)$, has finitely many pre-images in $U$.  Hence  $f_\lambda:U \rightarrow U_1$ is proper (by Theorem 1, \cite{bolsch}). 
Since $U$ and $U_1$ are simply connected (by Remark \ref{cifc-basic}(2)), it follows from the Riemann-Hurwitz formula (Lemma \ref{RH}) that deg $(f_\lambda)|_U =N+1$ where $N$ is the number of critical points of $f_\lambda$ in $U$ counting multiplicity. Since all the critical points of $f_\lambda$ are simple (Remark \ref{criticalpoints-rem}), the number $N$ here is in fact the number of distinct critical points.
\par 
If $U$ contains two critical points then it contains all the critical points (as the Fatou set is $\pi$-invariant and any two consecutive critical points are with the same imaginary part but with real parts differring by $\pi$ (See Lemma~\ref{criticalpoint})) and becomes horizontally spread. Therefore $N=0$ or $1$, and the degree $d$ of $f_\lambda: U \to U_1$ is  $1$ or $2$ respectively.

\end{enumerate}

\end{proof}
\begin{rem}\label{horizontally-bounded}
	If an unbounded Fatou component $U$ is not horizontally spread then $\{\Im(z): z \in U\}$ is unbounded but for every $z_0\in U$,  $\{\Re(z): z \in U ~\mbox{and}~\Im(z)>\Im(z_0)\}$ is bounded. 
\end{rem}
For proving Theorem~\ref{attractingdomaincomplete}, we also need the following.
\begin{lem}\label{nowanderingTHM}
Let $f_\lambda$ be  a topologically hyperbolic map for some $\lambda$ and $f_\lambda$ have a wandering domain $W$. Then there is an $n \geq 0$ such that  $W_n\cap P(f_\lambda) \neq \emptyset$. 
\end{lem}

\begin{proof}

Suppose on the contrary that  $W$ is a wandering domain of $f_\lambda$ such that $W_n\cap P(f_\lambda)=\emptyset$ for all $n \geq 0$. Since $f_\lambda$ is topologically hyperbolic, it follows from Lemma \ref{wandeingarbitrarylagredisc} that there exists an $n_0$ such that for all $n\geq n_0$, $W_n$ contains a disc of radius $ \pi$. In particular, $W_n$ contains a horizontal  line segment including its end point with length $\pi$. Since the Fatou set $\mathcal{F}(f_\lambda)$ is $\pi-$invariant (Lemma \ref{symmetry}),  $W_{n}$ contains a horizontal line unbounded in both the directions for all $n \geq n_0$. The horizontal strip bounded by  two such lines $l_{n_0} \subset W_{n_0}$ and  $l_{n_{0} +1} \subset W_{n_0 +1}$ contains a point of the Julia set, namely a point on the boundary of $W_{n_0}$.
 It follows from the fact $\mathcal{J}(f_\lambda)=\partial B$ (by Remark~\ref{cifc-basic}(1)) that  this strip contains a point of $B$. Now  $l_{n_0} \bigcup l_{n_{{0} +1}} \bigcup \{\infty\}$ is a closed curve in $\widehat{\mathbb{C}} \setminus B$ separating $B$. However, this is not posssible as $B$ is connected.\end{proof}

\begin{proof}[Proof of Theorem \ref{attractingdomaincomplete}]
 It follows from Lemma \ref{multiplierofthefixedpoints} that for $|2+\lambda^2|<1$, $f_\lambda$ has infinitely many attracting fixed points. The attracting domains corresponding to these attracting fixed points are distinct.
 \par 
 
 The point at $\infty$ is the only asymptotic value of $f_\lambda$ and is in the Julia set.    
 It follows from the proof of Lemma~\ref{symmetry} that if $c$ is a critical point such that $f_\lambda^n(c)$ converges to an attracting fixed point $z_0$ then $\lim_{n \to \infty}f_\lambda^n(c+k\pi)= z_0+k\pi$ for each $k \in \mathbb{Z}$. Recall that $z_0 +\pi k $ is an attracting fixed point if and only if $z_0$ is so. Note that every critical point of $f_\lambda$ in the lower half-plane is of the form  $ c+k\pi$ for some $k \in \mathbb{Z}$. Since each invariant attracting domain contains a critical point, each critical point in the lower half-plane is in an invariant attracting domain. Further, the distance between the forward orbit of each such critical point from the Julia set is the same and that is a positive real number. Hence  $d(\overline{P(f)}\cap  H^-, \mathcal{J}(f_\lambda)\cap \mathbb{C})>0$. It follows from Remark \ref{cifc-basic}(3) that $d(\overline{P(f)}\cap  H^+, \mathcal{J}(f_\lambda)\cap \mathbb{C})>0$.  Thus $f_\lambda$  is a topologically hyperbolic map for $|2+\lambda^2|<1$.
 
\begin{enumerate}

\item    Let $U$ be an invariant attracting domain. If $U$ is horizontally spread then $f_\lambda$ has an invariant Baker domain $\tilde{B}$ containing a lower half-plane by Lemma~\ref{proper}. 
If  there is a $\delta<0$ such that $ \Im(f_{\lambda}^{n_k}(z)) > \delta$ for some subsequence $n_k$ and for some $z \in \tilde{B}$ then the topologically hyperbolicity of $f_\lambda$ gives that $\tilde{B}$ contains the disc $D(z_{n_k}, |\delta|) $ for a sufficiently large $k$ (by Lemma~\ref{wandeingarbitrarylagredisc}). However  $D(z_{n_k}, |\delta|) $ contains a real number and that is either a pole or belongs to $B$. None of this can be true. Therefore,
\begin{equation} \label{top}
 ~\mbox{for each}~\delta<0~\mbox{ there is an }~n_\delta~\mbox{ such that }~\Im(f_{\lambda}^n (z)) \leq \delta~\mbox{ for all }~n >n_\delta. 
 \end{equation}

Now, 
choose a suitable  $\delta_0 <0$ such that $\Im(\tan z ) > -\Im(\lambda)$ for all $z$ with $\Im(z) < \delta_0$. This is possible because $\tan z \to -i$ as $\Im(z) \to -\infty$ and $-\sqrt{3}< -\Im(\lambda) <-1$. For such a $z$, let $z_n=f_{\lambda}^n(z)$ and observe that $\Im(z_1) >\Im(z)$. If $n_0$ is such that $\Im(z_n)<\delta_0$ for all $n >n_0$ then $\{\Im(z_n)\}_{n>n_0} $ is strictly increasing and bounded above by $\delta_0$. This sequence converges to some number less than or equal to $\delta_0$, which is a contradiction to Equation(\ref{top}) for a  $\delta< \delta_0$. Thus, the attracting domain $U$ is not horizontally spread. 
\par By Lemma~\ref{proper}(2), $f_\lambda: U \to U$ is a proper map of degree $1$ or $2$.  Since $U$ contains exactly one critical point of $f_\lambda$  by Lemma~\ref{singularvalues}, it follows from the Riemann-Hurwitz formula that the degree of $f_\lambda: U \to U$ is $2$.
 
\par
It follows from Theorem~\ref{accesstoinfty} that the number of invariant accesses from $U$ to its boundary points is $1,2$ or $3$. Further, each of these boundary points is either a weakly repelling fixed point or $\infty$. Since $f_\lambda$ has no weakly repelling fixed point, all these accesses are to $\infty$. Now, if there are more than one access to $\infty$ then for two curves $\gamma_1, \gamma_2$ in $U$ with a common starting point and landing at $\infty$, each component of $\hC \setminus (\gamma_1 \cup \gamma_2)$ would intersect the boundary of $U$. This is not possible as $\partial U \subset \partial B$. Thus there is exactly one invariant access from $U$ to $\infty$. In particular, $U$ is unbounded.

\par As $U$ is unbounded but not horizontally spread, it follows from Remark~\ref{horizontally-bounded} that $\{\Im(z): z \in U\}$ is unbounded but for every $z_0\in U$,  $\{\Re(z): z \in U ~\mbox{and}~\Im(z)>\Im(z_0)\}$ is bounded.

\item 
 The existence of any attracting domain  with period more than $1$ or any parabolic domain is  ruled out by Lemma~\ref{singularvalues}. Also by the same lemma,  $f_\lambda$ has neither any Siegel disc nor any Herman ring. The non-existence of any Baker domain (other than $B$)  or any wandering domain remains to be looked into. 

\par
Let $V$ be a $p$-periodic Baker domain of $f_\lambda$ such that $\lim_{n \to \infty}f^{np}_\lambda(z) = \infty$ uniformly  on $  V$. Since $f_\lambda$ is topologically hyperbolic, it follows from Lemma \ref{wandeingarbitrarylagredisc} that $V$ contains a disc of radius more that $\pi$. Since $\mathcal{F}(f_\lambda)$ is $\pi$-invariant (Lemma \ref{symmetry}), $V$ contains a  horizontal line which is unbounded in both the directions. This line separates $\mathbb{C} \cap \partial B$ from the boundary of each invariant attracting  domain since $\{\Im(z): z\in U\}$ is unbounded. Again $\partial V \subset \partial B$ implies that $V$ contains a half-plane of the form $\{z: \Im(z) < M<0\}$. But this is not true as there is a sequence of points in the invariant attracting domain whose imaginary parts tends to $-\infty$. Thus $f_\lambda$ does not have any Baker domain. 

\par 
There cannot be any wandering domain of $f_\lambda$  by Lemma \ref{nowanderingTHM}. 

\par 
\end{enumerate}
\end{proof}

\begin{rem}\label{translationofad}(Translation of attracting domains)
	For $|2+\lambda^2|<1$, the set of all (invariant) attracting domains of $f_\lambda$ can be enumerated in a specific way. In fact, if $V$ is an attracting domain of $f_\lambda$ corresponding to an attracting fixed point $z_0$ then for each $k\in \mathbb{Z}$, $V+k\pi$ is an invariant attracting domain corresponding to the attracting fixed point $z_0+k\pi$. Note that if $f^n_\lambda(z)\rightarrow z_0$ for $z\in V$ then $f^n_\lambda(z+k\pi)\rightarrow z_0+k\pi$. 
\end{rem}

Now the proof of Theorem~\ref{repelling-one} is presented.

\begin{proof}[Proof of Theorem \ref{repelling-one}]
\begin{enumerate}
\item Let $0< \Im(\lambda)<1$.
Since $\lim_{\Im(z) \to -\infty} \Im(\tan z)=-1$, choose $\delta<0$ such that the image of  $H_\delta= \{z: \Im(z)< \delta\}$ under $\tan z$ is contained in the half-plane $\{z: \Im(z) <\Im(-\lambda) \}$. This is also true for all smaller values of $\delta$. Then the image of $H_\delta$ under $z +\tan z$ is contained in $\{z: \Im(z) < \Im(-\lambda)+\delta\}$ and consequently, $f_{\lambda} (H_\delta) \subset H_{\delta}$. By the Fundamental Normality Test, the half-plane $H_\delta$ is contained in the Fatou set of $f_\lambda$. The Fatou component containing $H_\delta$, call it   $\tilde{B}$, is invariant. This Fatou component $\tilde{B}$ is simply connected by Remark~\ref{cifc-basic}(1). In particular, it is not a Herman ring. If an invariant Fatou component is a  Siegel disc, an attracting domain  or a parabolic domain then its closure contains a non-repelling fixed point. 
Since all the fixed points of $f_\lambda$ are repelling by Remark~\ref{lambda}, $\tilde{B}$ can neither be a Siegel disc, an attracting domain nor a parabolic domain. Thus,  $\tilde{B}$ is an invariant Baker domain.

Note that each critical point with positive imaginary part is contained in $B$.

 Let $\lambda=k \pi +i \lambda_2$ for some  $k \in \mathbb{Z}$ and $\lambda_2 >0$. If  $m$ is an integer then
 \begin{equation}\label{verticalline-2}
  f_{\lambda} (  m \pi +\frac{\pi}{2} +iy ) =k\pi   + m \pi  +\frac{\pi}{2}+i(\lambda_2+y+\coth y).
 \end{equation}  
 Here  $0< \lambda_2 <1$. 
 Let $L_{m \pi} =\{m\pi+ \frac{\pi}{2}+iy: y<0\}$  and $L_{(m+k)\pi}=k \pi +L_{m \pi}$. Then $f_\lambda (L_{m \pi}) \subset L_{(m+k)\pi}$ and $f^n _{\lambda}(L_{m\pi}) \subseteq L_{(m+kn)\pi}$ for all $n \geq 1$. For every $m$ and $z \in L_{m\pi}$, the sequence of real parts of $f_{\lambda}^n(z)$ tends to $\infty$ (or $-\infty$) as $n \to \infty$ when $k>0$ (or $k<0$ respectively).  We are to show that, \begin{equation}\label{claim}
 \lim_{n \to \infty} \Im(f_{\lambda}^n (z)) = -\infty ~\mbox{ for~ every}~ z \in L_{m\pi}. \end{equation}
  
For this, consider $\phi: (-\infty, 0) \to (-\infty, 0)$ defined by $\phi(y) =\lambda_2 +y +\coth y$ where $0<\lambda_2<1$. It is clear that  for $z \in L_{m \pi}$, $\Im(f_{\lambda} ^2 (z))=\phi^2 (y)$ and in general $\Im(f_{\lambda}^n (z)) = \phi^n(y) $ for all $n>0$.  Our claim (\ref{claim}) will be proved by showing that $\lim_{n \to \infty} \phi^n (y)=-\infty$ for all $y<0$. Since $\phi'(y) =1- \text{cosech}^2 (y)$, $\phi$ has a unique critical point and that is $y_0=-\sinh^{-1}1$. Further, it increases in $(-\infty, y_0)$, attains its maximum at $y_0$ and then decreases. Note that 
 $\lim_{y \to 0^-} \phi(y)=-\infty=\lim_{y \to -\infty} \phi(y)$. The image of $(-\infty, 0)$ under $\phi$ is strictly contained in $(-\infty, \phi(y_0))$. Since $\lambda_2 +\coth y_0<0, ~\phi(y_0)=\lambda_2+y_0+\coth y_0< y_0$ and  $\phi$ is strictly increasing in  $(-\infty, y_0)$, we have $\phi^{n}(y_0) \to -\infty$ as $n \to \infty$. This gives that $\lim_{n \to \infty} \phi^{n}(y) =-\infty$ for all $y<y_0$. Thus $\lim_{n \to \infty} \phi^{n}(y)=-\infty$ for all $y<0$.
 \par 
Since each critical point $c$ in the lower half-plane belongs to  $L_{m \pi}$ for some integer $m$, it follows from Equation (\ref{claim}) that $\lim_{n \to \infty} \Im(f_{\lambda}^{n}(c))=-\infty$. Note that $L_{m \pi} \subset \tilde{B}$ for all $m$ and in particular, $\tilde{B}$ contains  all the critical points (with negative imaginary part) and their forward orbits. Further, it is evident that $d(\overline{P(f)}\cap \tilde{ B}, \mathcal{J}(f_\lambda)\cap \mathbb{C})>0$.  This along with Remark \ref{cifc-basic}(3) prove that $f_\lambda$ is topologically hyperbolic. 
\par By the similar argument as used in Theorem \ref{attractingdomaincomplete}(2) and Lemma \ref{nowanderingTHM} we conclude that $f_\lambda$ does not have any non-primary periodic Fatou component other than $\tilde{B}$ or any wandering domain.

Every pole is of the form $m \pi +\frac{\pi}{2}$ and is an end point of $L_{m\pi}$ for some $m$. This gives that the boundary of $\tilde{B}$ contains a pole. As $\tilde{B}$ is simply connected, the Julia component (i.e., a maximally connected subset of the Julia set) containing a pole is unbounded.  If there  is a multiply connected Fatou component $V$ of a general meromorphic function then consider  a Jordan curve which is not contractible in $V$.  Arguing as in Lemma~1(\cite{tk-zheng}),  one finds that some iterated (forward) image of this curve surrounds a pole. This means that there is a bounded Julia component  containing a pole,  which is not possible. Thus the primary Fatou component $B$ and hence all the Fatou components are simply connected. Therefore, the Julia set of $f_\lambda$ is connected whenever $\lambda =k \pi  +i \lambda_2$ for $0< \lambda_2 <1$.

\item
By Lemma~\ref{criticalpoint}, the critical values of $f_\lambda$ corresponding to the critical points in the lower half-plane are $\lambda+\frac{\pi}{2}+n\pi -i(\sqrt{2}+\sinh^{-1}1)$ where $n$ is an integer.
For $\Im(\lambda)  > \sqrt{2}+\sinh^{-1} 1 \approx 2.295$, the imaginary part of each such critical value is non-negative. Hence all these critical values are in the primary component $B$. Thus $B$ contains all the critical values of the function. Consequently, there is no attracting domain, parabolic domain, Siegel disc or Herman ring in the Fatou set of $f_\lambda$ by Lemma~\ref{singularvalues}. 
\par
Clearly, $f_\lambda$ is topologically hyperbolic.  By Lemma~\ref{nowanderingTHM}, $f_\lambda$ has no wandering domain.
 Let  $f_\lambda$ have a non-primary $p-$periodic Baker domain and $z$ be a point in it.  Without loss of generality assume that  $\lim_{n \to \infty} z_n =\infty$ where $z_n =f_{\lambda}^{np} (z)$. If  there is a $\delta<0$ such that $ \Im(z_{n_k}) > \delta$ for some subsequence $n_k$ then the topologically hyperbolicity of $f_\lambda$ gives that the assumed Baker domain contains the disc $D(z_{n_k}, |\delta|) $ for a sufficiently large $k$ (by Lemma~\ref{wandeingarbitrarylagredisc}). However  $D(z_{n_k}, |\delta|) $ contains a real number and that is either a pole or belongs to $B$. None of this can be true. Therefore, for each $\delta<0$ there is an $n_0$ such that $\Im(z_n) < \delta$ for all $n >n_0$. In other words,  \begin{equation}\label{limit}
 \Im(z_{n}) \to -\infty ~\mbox{as}~n \to \infty.
 \end{equation} 
   Now, choose a sufficiently large $n_0$ such that $\Im(\tan z_n) > -2$ for all $n >n_0$. This is because $\tan z \to -i$ as $\Im(z) \to -\infty$. Since $\Im(\lambda)>2$, we have  $\Im(z_{n+1})=\Im(\lambda)+\Im(z_n)+\Im(\tan z_n) >\Im(z_n)$ for all $n>n_0$.  This is a contradiction to (\ref{limit}).
 Thus $f_\lambda$ does not have any Baker domain. 
  Therefore $B$ is the only Fatou component of $f_\lambda$ for $\Im(\lambda)  >\sqrt{2}+ \sinh^{-1} 1 $.
 
 \par
That the Julia set is disconnected will be established by proving the existence of a bounded component of the Julia set. This is because $\infty \in \mathcal{J}(f_\lambda)$.
 This desired Julia component is going to be the one containing a pole of the function.
 \par
 Since the Fatou set is connected, no Julia component separates the plane, i.e., its complement is connected.
 Let $J$ be a connected subset of $\mathcal{J}(f_\lambda) \cap \mathbb{C}$ containing a pole. If $J$ contains another pole then by Lemma~\ref{symmetry}, it contains all the poles of $f_\lambda$ and then it separates the plane. However this is not possible implying that $J$ contains exactly one pole, say $z_0$. Let $J_0 $ be a connected subset  of $J \setminus \{z_0\}$. Then $f_\lambda (J_0) \subset H^-$ and all the critical values of $f_\lambda$ are in $H^+$. Take a point $z' \in J_0$ and consider a branch $g$ of $f_{\lambda}^{-1}$ defined in a neighborhood of $f_{\lambda}(z')$  such that  $g(f_\lambda (z'))=z'$. This $g$ can be analytically continued to the whole of $H^{-}$ by the Monodromy theorem. In particular, $g$ is analytically defined in a simply connected domain in $H^-$ containing $f_{\lambda}(J_0)$. In other words, the function $f_\lambda$ is one-one on $J_0$. 
 \par Now assuming that $J_0$  is unbounded, consider two connected disjoint subsets $J_{z_0}$ and $J_\infty$ of $J_0$ containing $z_0$ and $\infty$ in their closures respectively. Observe that $f_{\lambda} (J_{z_0}) $ and $ f_{\lambda} (J_\infty) $ are both unbounded and connected. Further  $f_{\lambda} (J_{z_0})\cap f_{\lambda} (J_\infty) =\emptyset$. Now $f_{\lambda}(J_0)$ is a conected subset of $\mathcal{J}(f_\lambda) \cap \mathbb{C}$ containing two disjoint  and connected subsets, each of which is unbounded. Thus $f_{\lambda}(J_0)$ and hence the Julia component containing it, separates the plane. This is a contradiction.  This proves that every  connected subset  of $J \setminus \{z_0\}$ is bounded.  Therefore $J$ is bounded and the proof completes.
\end{enumerate}
 \end{proof}
 Here is a remark on $f_\lambda$ for $\lambda$ with real part different from any integral multiple of $\pi$.
 \begin{rem}
For $0<\Im(\lambda)<1$, consider the critical point $c_0=\frac{\pi}{2}-i \sinh^{-1}1 \in \overline{C}$ of $f_\lambda$. Now $f_\lambda(c_0)=\lambda +c_0+\tan(\frac{\pi}{2}-i \sinh^{-1}1)=\lambda +c_0+\cot(i \sinh^{-1}1)=\lambda +c_0-i\coth(\sinh^{-1}1)$. This gives that $\Im(f_{\lambda}(c_0)) =\Im(\lambda)+\Im(c_0)-\coth(\sinh^{-1}1)$.  Since $\coth(x)>1$ for all $x>0$, $\Im(\lambda)-\coth(\sinh^{-1}1)<0$ giving that $\Im(f_\lambda(c_0)) < \Im(c_0)$. It now follows from Lemma~\ref{criticalpoint} and Lemma~\ref{symmetry} that   $\Im(f_\lambda(c)) < \Im(c)$ for all $c \in \overline{C}$. However, this argument seems to fail to conclude anything about the iterated images of the critical values.
 \end{rem}

%
%
%
%

In addition to give a proof of  Theorem \ref{wandeing domain} theoretically, we also provide a proof by numerically estimating the regions contained in the wandering domains.\\
\begin{proof}[Proof of Theorem \ref{wandeing domain}] For $\lambda_0=i\frac{\pi}{2}$, we have  $|2+\lambda_{0}^2|<1$. By Theorem~\ref{attractingdomaincomplete}, the Fatou set of  $f_{\lambda_0}$ contains infinitely many invariant attracting domains along with the primary Fatou component $B$. For each natural number $k$ and $\lambda_k=k\pi+\lambda_0$, $f_{\lambda_k}(z)=k\pi+\lambda_0+z+\tan(z)=k\pi+f_{\lambda_{0}}(z),~f^2_{\lambda_k}(z)=k\pi+\lambda_0+k\pi+f_{\lambda_{0}}(z)+\tan(k\pi+f_{\lambda_{0}}(z))=2k\pi+f^2_{\lambda_{0}}(z)$, and in general, 
\begin{equation}\label{new}
f^n_{\lambda_k}(z)=nk\pi+f^n_{\lambda_{0}}(z) ~\mbox{for all}~ n\geq 1 ~\mbox{and for all} ~z\in \mathbb{C}.
\end{equation}
This implies that $z$ is contained in the Fatou set of $f_{\lambda_{0}}$ if and only if it is in the Fatou set of $f_{\lambda_{k}}$ for each natural number $k$. 

Though the Fatou set of $f_{\lambda_k}$ and that of $f_{\lambda_{0}}$ are identical, their iterative behaviour on their respective Fatou components differ. It follows from Equation (\ref{new}) that the primary Fatou component of $f_{\lambda_{k}}$ is the same as that of $f_{\lambda_{0}}$ (see Theorem \ref{cifcupperhalf}) and Remark \ref{cifc-basic}(3) holds for $f_{\lambda_k}$. Now we discuss the nature of an attracting domain of $f_{\lambda_{0}}$ as a Fatou component of $f_{\lambda_k}$. Let $V$ be an invariant attracting domain of $f_{\lambda_{0}}$ corresponding to a fixed point $z_0$ and $z\in V$. Then 
\begin{equation}\label{escaping}
f^n_{\lambda_{0}}(z)\rightarrow z_0 ~\mbox{and}~  f^n_{\lambda_k}(z)=nk\pi+f^n_{\lambda_{0}}(z)\rightarrow \infty ~\mbox{as}~ n\rightarrow \infty. 
\end{equation}
Note that for each $n$, $f^n_{\lambda_k}(z)$ is contained in different Fatou components (Remark \ref{translationofad}). Thus $V$ is an wandering domain of $f_{\lambda_k}$. As $f_{\lambda_k}$ maps $V$ into $V+k\pi$ (Remark \ref{translationofad}), $V+\pi$ is mapped into $V+(k+1)\pi$. Hence $f_{\lambda_k}$ has $k$ many wandering domains with different grand orbits, namely $V, V+\pi, V+2 \pi,\cdots, V+(k-1)\pi$. 
\begin{enumerate}
	\item It follows from Equation (\ref{escaping}) that these wandering domains are escaping.
	\item Since  $V+n\pi$ for  each $n \in \mathbb{Z}$, contains a critical point of $f_{\lambda_0}$ and the critical points of $f_{\lambda_{0}}$ and that of $f_{\lambda_{k}}$ are the same, each such wandering domain $V+n\pi$ contains a critical point. It cannot contain more than one critical point as each two critical point are separated by a vertical line contained in the primary Fatou component $B$. In fact, it follows from Lemma~\ref{lines}(2) that for  all $z \in l_{m \pi}$, $\lim_{n \to \infty} \Im(f_{\lambda_k}^n (z))=+\infty$  and we have   $l_{m\pi}  \subset B $ for all integers $m$. This gives that no $V+ n\pi$ is horizontally spread for any integer $n$. By Lemma~\ref{proper}(2), $f_{\lambda_k}: V+n \pi \to V + (n+k) \pi$ is proper. Its degree is $2$ by the Riemann-Hurwitz formula. Taking $ W_n =V+n \pi$ and $W_{n+1}=V+ (n+k) \pi$, it is seen that $f_{\lambda_k}: W_n \to W_{n+1}$ is a proper map with degree $2$.
	\item  If $W'$ is a wandering domain in the grand orbit of $W(=W_0=V)$ and is different from each $W_n$ then there is no critical point in $W'$ and the map $f_{\lambda_k}$ is one-one on $W'$ by the Riemann-Hurwitz formula.
\end{enumerate} 
It follows from Equation~(\ref{escaping}) that if $c$ is a critical point of $f_{\lambda_k}$ (this is also a critical point of $f_{\lambda_0}$) in the lower half-plane and $f_{\lambda_0}^n (c) \to z_0$ as $n \to \infty$ then $f_{\lambda_k}^n (c) \to \infty$ and $|f_{\lambda_k}^n (c)-\pi n k- z_0| \to 0$ as $n \to \infty$. This is true for all the critical points of $f_{\lambda_k}$ lying in the lower-half-plane by Lemma~\ref{criticalpoint} and Remark~\ref{translationofad}. In other words, though the forward orbits of all such critical points  tend to $\infty$, they do so by maintaining a positive distance from $\mathcal{J}(f_{\lambda_k}) \cap \mathbb{C}$. Thus
$f_{\lambda_k}$ is topologically hyperbolic. Using similar argument as used in Theorem \ref{attractingdomaincomplete}, it can be shown that $f_{\lambda_k}$ does not have any periodic Fatou component except $B$ or any other wandering domain. This completes the proof.
\end{proof}

Now we proceed to prove Theorem \ref{wandeing domain} by estimating the region of the wandering domains. To prove Theorem \ref{wandeing domain} we need two lemmas.

\begin{lem}

\begin{enumerate}\label{invariantlines}
If $\lambda= k\pi+i\frac{\pi}{2}$ for a non-zero integer $k$ then the following are true. 
\item 
 The vertical line $l_{m\pi}=\{z:\Re(z)=m\pi\}$ is contained in $B$ for all integers $m$. 

\item 
The vertical half-line  $l_{m\pi+\frac{\pi}{2}}=\{z:\Re(z)=m\pi+\frac{\pi}{2} ~\mbox{and}~ -\infty < \Im(z)\leq -\sinh^{-1}1\}$ is mapped into the half-line $l^{-}_{(m+k)\pi+\frac{\pi}{2}}=\{z:\Re(z)=(m+k)\pi+\frac{\pi}{2} ~\mbox{and}~ \Im(z)<0\}$ for all integers $m$. 

\item 
None of the critical points in the lower half-plane is contained in $B$.

\end{enumerate}
\end{lem}

\begin{proof}
\begin{enumerate}
\item 
It follows from Lemma~\ref{lines}(2) that for  all $z \in l_{m \pi}$, $\lim_{n \to \infty} \Im(f_{\lambda}^n (z))=+\infty$. We are done since $l_{m\pi} \cap B \neq \emptyset$ for all integers $m$. 

\item 
For $z \in l_{m\pi+\frac{\pi}{2}}$,  $f_\lambda(z)=k\pi+i\frac{\pi}{2}+m\pi+\frac{\pi}{2}+i\Im(z)+\tan(m\pi+\frac{\pi}{2}+i\Im(z))=(k+m)\pi+\frac{\pi}{2}+i\{\Im(z)+\frac{\pi}{2}+\coth\Im(z)\}$. Note that $-\infty<\coth\Im(z)<-1$ for all $z$ with  $\Im(z)<0$. 
Therefore, $\Im(f_\lambda (z))<-\sinh^{-1}1 + \frac{\pi}{2}-1 <0$ for all $z \in l_{m \pi + \frac{\pi}{2}}$.
 Thus $f_\lambda$ maps $l_{m\pi+\frac{\pi}{2}}$ into $l^{-}_{(m+k)\pi+\frac{\pi}{2}}$.

\item 
The critical points of $f_\lambda$ in the lower half-plane are $m\pi+\frac{\pi}{2}-i\sinh^{-1}1$ for $m \in \mathbb{Z}$. For each  $z \in l_{m\pi+\frac{\pi}{2}}=\{z:\Re(z)=m\pi+\frac{\pi}{2} ~\mbox{and}~\Im(z)\leq -\sinh^{-1}1\}$, $\Re(f_\lambda(z))=(m+k)\pi+\frac{\pi}{2}$
and $\Im(f_\lambda(z))=\frac{\pi}{2}+\Im(z)+\coth(\Im(z))$.
\par
Consider the function $g(y)=\frac{\pi}{2}+y+\coth y$, $y<0$ and $h(y)=g(y)-y$. Note that $\lim_{y \to -\infty}h(y)=\frac{\pi}{2}+\lim_{y \to -\infty}\coth y=\frac{\pi}{2}-1>0$ and $\lim_{y \to 0^-}h(y)=-\infty$. By the Intermediate Value Theorem, there is a negative real number $y_0$ such that $h(y_0)=0$. This $y_0$ is a fixed point of $g(y)$. Since $h'(y)=-\text{cosech}^2 y<0$ for all $y<0$, $y_0$ is unique. Note that $y_0=\frac{1}{2}\ln\frac{\pi-2}{\pi+2}\approx -0.7524$.  Note that $g'(y)=2-\coth^2y$. The multiplier of $y_0$, $g'(y_0)= 2-\coth^2 y_0=2-\frac{\pi^2}{4} \in (-1,0)$ which means that $y_0$ is an attracting. Now $g''(y)=2\coth y ~\text{cosech}^2y<0$ for all $y<0$. Then $g'$ has a unique root and,

\begin{gather}{\label{derivative}} 
 g'(y)
  \begin{cases}
  >0~~~~~~~~ for~ all ~ y< -\sinh^{-1}1\approx -0.8814,\\
  =0~~~~~~~~ for~~ ~y= -\sinh^{-1}1,\\
  <0 ~~~~~~~~ for~ all ~ -\sinh^{-1}1< y<0.\\
  \end{cases}
\end{gather}
 Note that $g([-0.8814,y_0])=[y_0,-0.7248]$ and $g(-0.7248)> -0.8814$. Since $g$ is decreasing in $(-0.8814,0)$, $g([y_0,-0.7248])\subsetneq [-0.8814,y_0]$ and it follows that $g^{n+1}([-0.8814,-0.7248])\subsetneq g^{n}([-0.8814,-0.7248])$ for all $n$. Thus $g^n(y) \rightarrow y_0$ for all $y\in [-0.8814,-0.7248]$.
\par
Since the image of $l_{m\pi+\frac{\pi}{2}}$ is $l^{-}_{(m+k)\pi+\frac{\pi}{2}}$ under $f_\lambda$, $\Im(f_\lambda^n(z))=g^n(\Im(z))$ for all $z\in l_{m\pi+\frac{\pi}{2}}$ and all $n$. Let $c_0=\frac{\pi}{2}-i\sinh^{-1}1$. Note that $f_\lambda(c_0)=k\pi+i\frac{\pi}{2}+\frac{\pi}{2}-i\sinh^{-1}1+\tan(\frac{\pi}{2}-i\sinh^{-1}1)=k\pi+\frac{\pi}{2}+i\{\frac{\pi}{2}-\sinh^{-1}1-\coth(\sinh^{-1}1)\}=k\pi+\frac{\pi}{2}-0.7248i$. As $\Im(f_\lambda(c_0))\in [-0.8814,-0.7248]$ then $\Im(f_\lambda^n(c_0))\rightarrow y_0\approx -0.7524$ and henec $c_0$ is not contained in $B$. It follows from Lemma~\ref{symmetry} that none of the critical points in the lower half-plane is contained in $B$. 

\end{enumerate}
\end{proof}

Here are some estimates of three functions in suitable intervals. 

\begin{lem}\label{estimates}
\begin{enumerate}
	\item 	\label{realhx}
	If $x\leq -0.6658$ then $ 0< \frac{\sin\frac{\pi}{8}}{-\cos\frac{\pi}{8}+\cosh2x} <\frac{\pi}{8}$.
	\item \label{hxmaximum}
	For all $x\leq -0.6658$,  $ \frac{\pi}{2}+x+\frac{\sinh 2x}{-\cos\frac{\pi}{8}+\cosh 2x} \leq-0.6658$.
	\item \label{lastlemma}
	If $m$ is an integer and $|x-(m\pi +\frac{\pi}{2})| \leq \frac{\pi}{16}$ then  $\frac{\pi}{2}-0.6658-\frac{\sinh1.3316}{\cos2x+\cosh1.3316} <-0.6658$.
\end{enumerate}

\end{lem}

\begin{proof}
	\begin{enumerate}
\item For $h(x)=\frac{\sin\frac{\pi}{8}}{-\cos\frac{\pi}{8}+\cosh2x}$, $h'(x)=-2\sin\frac{\pi}{8}\frac{\sinh2x}{(-\cos\frac{\pi}{8}+\cosh2x)^2} >0$ for all $x<0$. The function $h$ is strictly increasing. Further, $\lim_{x \to -\infty}h(x)=0$ and $h(-0.6658)\approx0.3473$. This  gives that $0<h(x)\leq 0.3473<\frac{\pi}{8}$ for $x\leq -0.6658$.
\item 
Let $h(x)=\frac{\pi}{2}+x+\frac{\sinh 2x}{-\cos\frac{\pi}{8}+\cosh 2x} $. Then
$h'(x)=1+2\frac{1-\cosh 2x\cos\frac{\pi}{8}}{(-\cos\frac{\pi}{8}+\cosh 2x)^2}$ and  $h''(x)=4\sinh 2x\frac{\cos^2\frac{\pi}{8}+\cos\frac{\pi}{8}\cosh2x-2}{(-\cos\frac{\pi}{8}+\cosh 2x)^3}$. The function $\cos^2\frac{\pi}{8}+\cos\frac{\pi}{8}\cosh2x-2 $ is a strictly decreasing function with its minimum value approximately equal to  $0.7249$ achieved at $-0.6658$. Thus  $h''(x)<0$ for all $x \leq -0.6658$ giving that $h'$ is a strictly decreasing function. As
$\lim_{x \to -\infty}h'(x)=1$ and $h'(-0.6658)\approx -0.4359$, there exists a unique $x_0 \leq -0.6658$ such that $h'(x_0)=0$. Computationally, it is found that $x_0\approx -0.804$.  This proves that $h$ attains maximum at $x_0$ and the maximum value is $\approx -0.6658$. Thus $h(x)\leq -0.6658$ for all $x \leq -0.6658$.
\item 
Let $h(x)=\frac{\pi}{2}-0.6658-\frac{\sinh1.3316}{\cos2x+\cosh1.3316}$ for   $x \in I_m=\{x:|x-(m\pi +\frac{\pi}{2})| \leq \frac{\pi}{16}\}$.  Then $h'(x)=-2\sinh1.3316\frac{\sin2x}{(\cosh1.3316+\cos2x)^2}$ is $0$ only when  $x=m\pi+\frac{\pi}{2}$.
Further, $h'(x)<0$ for $x< m\pi +\frac{\pi}{2}$ and $h'(x)>0$ for $x>  m\pi +\frac{\pi}{2}$ giving that $h$ attains its minimum at $m\pi+\frac{\pi}{2}$.  As $h(m\pi+\frac{\pi}{2}-\frac{\pi}{16}) =h(m\pi+\frac{\pi}{2}+\frac{\pi}{16})\approx-0.6939$, we have  $h(x)<-0.6939 < -0.6658$ for all $x\in I_m$.

\end{enumerate}
\end{proof}

\begin{proof}[ Proof of Theorem \ref{wandeing domain} (by estimation)]
Let $\lambda= k\pi+ i \frac{\pi}{2}$ for a natural number $k$.
Firstly, we show that certain regions outside the primary Fatou component are in the Fatou set of $f_\lambda$. Consider the region $R_m=\{z: |\Re(z)-(m\pi +\frac{\pi}{2})|<\frac{\pi}{16} ~~\text{and}~\Im(z)\leq -0.6658\}$. Note that $R_m$ does not contain any pole of $f_\lambda$. Our intention is to show that  $f_\lambda(R_m)\subset R_{m+k}$. Let $$l_1=\{z:\Re(z)=m\pi+\frac{\pi}{2}-\frac{\pi}{16}~\text{and}~\Im(z)\leq -0.6658\},$$  $$l_2=\{z:\Re(z)=m\pi+\frac{\pi}{2}+\frac{\pi}{16}~\text{and}~\Im(z)\leq -0.6658\}$$ and $$l_3=\{z:| \Re(z)-(m \pi+\frac{\pi}{2})| \leq \frac{\pi}{16}~\text{and}~\Im(z)=-0.6658\}.$$
 The boundary of $R_m$ is $l_1 \cup l_2 \cup l_3 \cup \{\infty\}$.

For $z \in l_1$,  $\Re (f_\lambda(z))=(k+m)\pi +\frac{\pi}{2} -\frac{\pi}{16} +\Re(\tan z)=(k+m)\pi+\frac{\pi}{2}-\frac{\pi}{16}+\frac{\sin \frac{\pi}{8}}{-\cos\frac{\pi}{8}+\cosh 2\Im(z)}$ and $\Im (f_\lambda(z))=\frac{\pi}{2}+\Im(z)+\frac{\sinh 2\Im(z)}{-\cos\frac{\pi}{8}+\cosh 2\Im(z)}$. It follows from Lemma \ref{estimates}(\ref{realhx}) that $(m+k)\pi+\frac{\pi}{2}-\frac{\pi}{16}\leq \Re(f_\lambda(z))\leq (k+m)\pi+\frac{\pi}{2}+\frac{\pi}{16}$. Similarly,  Lemma \ref{estimates}(\ref{hxmaximum})  gives that  $\Im(f_\lambda(z))\leq-0.6658$  for all $z \in l_1$.

 Now, for $z \in l_2$,  $\Re (f_\lambda(z))=(k+m)\pi+\frac{\pi}{2}+\frac{\pi}{16}-\frac{\sin \frac{\pi}{8}}{-\cos\frac{\pi}{8}+\cosh 2\Im(z)}$ and $\Im (f_\lambda(z))=\frac{\pi}{2}+\Im(z)+\frac{\sinh 2\Im(z)}{-\cos\frac{\pi}{8}+\cosh 2\Im(z)}$.
  By Lemma \ref{estimates}( \ref{realhx}), $(k+m)\pi+\frac{\pi}{2}-\frac{\pi}{16}\leq \Re(f_\lambda(z))\leq (m+k)\pi+\frac{\pi}{2}+\frac{\pi}{16}$. Similarly, Lemma  \ref{estimates}(\ref{hxmaximum}) gives that $\Im(f_\lambda(z))\leq-0.6658$  for all $z \in l_2$ . 
  
 If $z \in l_3$ then  $\Im(f_\lambda(z))=\frac{\pi}{2}+\Im(z)+\Im(\tan z)=\frac{\pi}{2}-0.6658-\frac{\sinh1.3316}{\cos2x+\cosh1.3316}$. It follows from Lemma \ref{estimates}(\ref{lastlemma}) that $\Im(f_\lambda(z))<-0.6658$ for all $z \in l_3$. This strict inequality implies that the points $l_1\cap l_3$ and $l_2\cap l_3$ are mapped to the interior of $R_{m+k}$. Thus $f_\lambda(R_m)\subsetneq R_{m+k}$ (See Figure \ref{RMF}) and $\cup_{n \in \mathbb{N}}   R_{m+nk} $ is invariant under $f_\lambda$ giving that $R_m$ is in the Fatou set of $f_\lambda$ by the Fundamental Normality Test. Since $m$ is arbitrary, the Fatou set of $f_\lambda$ contains $R_m$ for every $m\in \mathbb{Z}$.
 
 Thus $f_\lambda(R_m)\subset R_{m+k}$ and $\cup_{n \in \Z}   R_{m+nk} $ is invariant under $f_\lambda$ giving that $R_m$ is in the Fatou set of $f_\lambda$ for every integer $m$ by the Fundamental Normality Test. 
 
  For each integer  $m$, the line $L_{m}=\{z:\Re(z)=m\pi+\frac{\pi}{2} ~\mbox{and}~ \Im(z)\leq -0.6658\}$ is  contained in $R_m$. Between any two such consecutive lines $L_m$ and $L_{m+1}$, there is a vertical line $l_{(m+1) \pi}$ which is in the primary Fatou component (by Lemma~\ref{invariantlines}(1)). In other words, for $m \neq m'$, the Fatou components containing $R_m$ is different from that containing $R_{m'}$.
  
  \begin{figure}[H]
  	\centering
  	\includegraphics[width=10cm,height=6cm,angle=0]{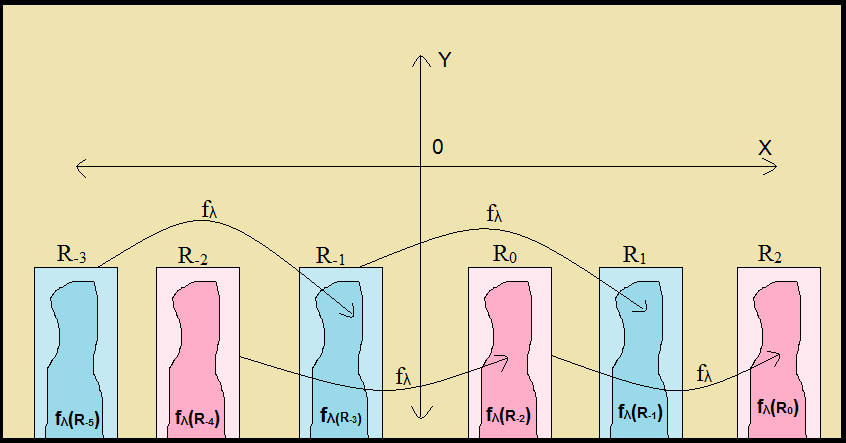}
  	\caption{Visualization of $R_m$'s for $k=2$.}
  	\label{RMF}
  \end{figure}

  Let $W$ be the Fatou component containing $R_0$. Then all the $W_n$s are distinct giving that $W$ is a wandering domain. 
  \begin{enumerate}
 \item Note that  $R_{nk}$ is in the Fatou set and is contained in $W_n$ for each $n$. Further, $f_{\lambda}^n \to \infty$ on $W$. Thus, $W$ is escaping.
 \item Since each $R_{nk}$ contains a critical point of $f_\lambda$, each $W_n$ contains a critical point. It cannot contain more than one critical point as each two critical point are separated by a vertical line contained in $B$.  For the same reason, no $W_n$ is horizontally spread. By Lemma~\ref{proper}(2), $f_\lambda: W_n \to W_{n+1}$ is proper. Its degree is $2$ by the Riemann-Hurwitz formula. Let, for a natural number $n$, $W_{-n}$ be the wandering domain containing $R_{-kn}$ such that $f^{n}_\lambda (W_{-n})=W$. The above argument gives that $f_\lambda: W_m \to W_{m+1}$ is proper map with degree $2$ for all negative integer $m$.
  	\item  If $W'$ is a wandering domain in the grand orbit  and is different from all $W_n$ then there is no critical point in $W'$ and the map $f_\lambda$ is one-one on $W'$ by the Riemann-Hurwitz formula.
  \end{enumerate}  
  It can be seen that, for $i\in \{1,2,\cdots k-1\}$, the Fatou component containing $R_i$ is also a wandering domain $\mathcal{W}^i$ and their forward orbits are disjoint from each other and also from $W$. Thus, there are $k$ wandering domains with distinct forward orbits. Clearly, their grand orbits are also different. 
  \par
  Note that $f_\lambda$ is topologically hyperbolic. Using similar argument as given in Theorem \ref{attractingdomaincomplete}(2), it can be shown that $f_\lambda$ does not have any periodic Fatou component except $B$ or any other wandering domain.

\end{proof}
\begin{rem}
For $k<0$, there are wandering domains $W$ with the same properties except that $\Re(f_{\lambda}^n) \to -\infty$ on $W$ as mentioned in Theorem~\ref{wandeing domain}.
\end{rem}
\section{Concluding remarks}
We first summarize the dynamics of $f_\lambda$ in terms of the parameter $\lambda$ for $\Im(\lambda)>0$ (Figure~\ref{parameterplane}). Since for every $\lambda$, $f_\lambda$ has a completely invariant Baker domain, the primary Fatou component, we describe the other Fatou components only. An archetype of the parameter plane is givend in Figure \ref{parameterplane}(a). Its computer generated version is given in Figure \ref{parameterplane}(b). The parameters in the strip $\{\lambda: 0<\Im(\lambda)<1\}$ (seen in yellow) correspond to $f_\lambda$ with an invariant Baker domain as mentioned in Theorem~\ref{repelling-one}. This is the only non-primary Fatou component if $\Re(\lambda)=k \pi$ whenever $k \in \mathbb{Z}$. 
The parameters in the yellow region $\{\lambda: |2+\lambda^2|<1\}$, we call this the attracting lobe, correspond to the existence of infinitely many invariant attracting domains as described in Theorem~\ref{attractingdomaincomplete}.
	For a fixed integer $k$, $f_{\lambda+k \pi}^n(z)=n k \pi+f_{\lambda}^n (z)$ for every natural number $n$ and $z \in \mathbb{C} $. If $|2+\lambda^2|<1 $ and $A_\lambda$ is an attracting domain of $f_\lambda$ then  $f^n_{k \pi+\lambda} \to \infty$ uniformly on each compact subset of $A_\lambda$. In other words, all the attracting domains of $f_\lambda$ are contained in the Fatou set of $f_{\lambda +k \pi}$. For $k \neq 0$, with some extra effort these attracting domains of $f_\lambda$ have been shown to be wandering domains for $f_{i \frac{\pi}{2}+k \pi}$ in Theorem~\ref{wandeing domain}. Further, since all the critical points in $H^-$ of $f_\lambda$ are in the invariant attracting domains, the function $f_{\lambda+k \pi}$ is topologically hyperbolic. The dynamics of $f_\lambda$ for other values of $\lambda$ is to be taken up later.
  The primary Fatou component is the only Fatou component of $f_\lambda$ and the Julia set is disconnected whenever $\lambda$ is in the yellow strip $\{\lambda: \Im(\lambda)> \sqrt{2}+\sinh^{-1}1\}$ above the attracting lobe.  This is given in Theorem~\ref{repelling-one}(2). It is important to note that the attracting lobe does not touch this strip. The situation for $f_\lambda$ is the same when $\Im(\lambda)= \sqrt{2}+\sinh^{-1}1$ but $\Re(\lambda) \neq k \pi+\frac{\pi}{2}, k \in \mathbb{Z}$. For $\lambda=k \pi+\frac{\pi}{2}+i(\sqrt{2}+\sinh^{-1}1 )$, the poles become the critical values and the function is no longer topologically hyperbolic. But the dynamics seems to be tractable!
%

\begin{figure}[H]
	\centering
	\subfloat[]
	{\includegraphics[width=2.85in,height=2.0in]{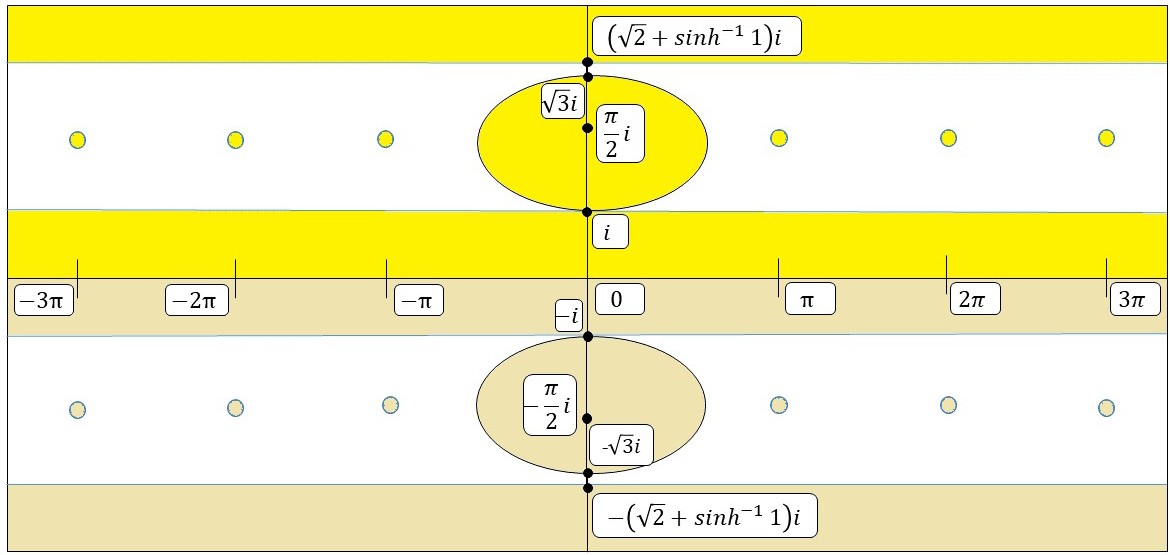}}
	\vspace{.2cm}
	\subfloat[]	{\includegraphics[width=2.75in,height=2.0in]{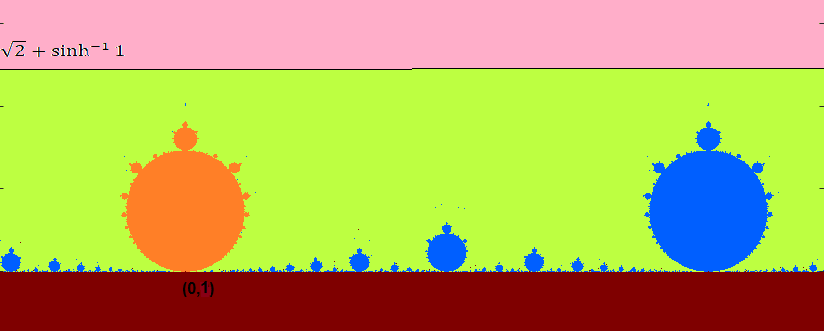}}
	\caption{The parameter plane}
	\label{parameterplane}
\end{figure}

 Some of the dynamically crucial properties of $f_\lambda$ are due to $\tan z$. In place of $\tan z$, one may consider a periodic meromorphic function $h$ such that $1+h'(z)=g(h)$ for an entire function $g$. If $F_\lambda (z)=\lambda+z+h(z)$ is such a function then the following are true.

\begin{enumerate}
	\item The function $F_\lambda$ has infinitely many fixed points for all except possibly two values of $\lambda$ and the multiplier of every fixed point is $g(-\lambda)$. To see it,  note that every fixed point $z_0$ of $F_\lambda$  satisfies  $h(z_0)=-\lambda$ and since $h$ is meromorphic, for all but atmost two values of $\lambda$, $h(z_0)=-\lambda$ has infinitely many solutions.
The multiplier of $z_0$ is $F_\lambda'(z_0)=1+h'(z_0)=g (h(z_0))=g (-\lambda)$.
\item The Fatou set (and therefore the Julia set) of $F_\lambda$ is $w$-invariant where $w$ is the period of $h$. This follows from the fact that $F^n_{\lambda}(z+w )=w+F^n_{\lambda}(z)$ for all $n$ and $z \in \mathbb{C}$. 
\item The set of  all the singular values of $F_\lambda$ is unbounded  whenever $g$ has at least three distinct roots. To see it, first note that  the critical points of  $F_\lambda$ are the solutions of $g(h(z))=0$. Since $g$ has at least three distinct roots, there is a solution of $g(h(z))=0$. If $g(h(c))=0$ for some $c$ then
for each $n \geq 0$, $g(h(c+nw))=g(h(c))=0$ and $c+nw$ is  a critical point of $F_\lambda$.  The corresponding critical values are $F_\lambda (c+n w)=\lambda+c+nw+h(c)$. We are done as the set $\{\lambda+c+nw+h(c): n \geq 0\}$ of critical values of $F_\lambda$ is unbounded. 
\end{enumerate}
 The dynamics of $F_{\lambda}$ can be studied possibly under some additional conditions on $h$.
 
 \section{Competing Interests}
 On behalf of all authors, the corresponding author states that there is no conflict of interest.
\bibliographystyle{amsplain}
\addcontentsline{toc}{chapter}{\numberline{}References}
 
\end{document}